\documentclass[12pt]{amsart}

\title[Quasi-modular spaces]{Quasi-modular spaces with applications to quasi-normed Calder\'{o}n--Lozanovski\u{\i}{} spaces}

\author[Pawe{\l} Foralewski]{Pawe{\l} Foralewski}
\address{Faculty of Mathematics and Computer Science, Adam Mickiewicz University, Pozna\'{n},
Uniwersytetu Pozna\'{n}skiego 4, 61-614 Pozna\'{n}, Poland}
\email{katon@amu.edu.pl}

\author[Henryk Hudzik]{Henryk Hudzik$^1$}\thanks{$^1$ This work was initiated and partially discused with Professor Henryk Hudzik who passed away on March 2, 2019. This is our tribute to our dear friend and master.}
\address{Faculty of Mathematics and Computer Science, Adam Mickiewicz University, Pozna\'{n},
Uniwersytetu Pozna\'{n}skiego 4, 61-614 Pozna\'{n}, Poland}

\author[Pawe{\l} Kolwicz]{Pawe{\l} Kolwicz$^2$}\thanks{$^2$ Corresponding author}
\address{Poznan University of Technology, Institute of Mathematics,  
Piotrowo 3A, 60-965 Pozna\'{n}, Poland}
\email{pawel.kolwicz@put.poznan.pl}

\usepackage[latin1]{inputenc}
\usepackage[english]{babel}
\usepackage{amsmath,amssymb,amsthm,amsfonts}
\usepackage[nobysame]{amsrefs}

\newtheorem{theorem}{\indent Theorem}[section]
\newtheorem{definition}[theorem]{\indent Definition}

\newtheorem{lemma}[theorem]{\indent Lemma}
\newtheorem{remark}[theorem]{\indent Remark}
\newtheorem{example}[theorem]{\indent Example}
\newtheorem{corollary}[theorem]{\indent Corollary}

\numberwithin{equation}{section}

\numberwithin{equation}{section}
\DeclareMathOperator{\supp}{supp}

\subjclass[2010]{46E30, 46A80, 46B20, 46A16}

\keywords{quasi-modular, quasi-normed ideal space, quasi-normed Calder\'{o}n--Lozanovski\u{\i}{} space, order isomorphic and order linear isometric copies of $l^{\infty}$}

\begin{document}

\begin{abstract}
In this paper we introduce the notion of a quasi-modular and we prove that the respective Minkowski functional of the unit quasi-modular ball becomes a quasi-norm. In this way, we refer to and complete the well-known theory related to the notions of a modular and a convex modular that lead to the $F$-norm and to the norm, respectively. We use the obtained results to consider basic properties of quasi-normed Calderón--Lozanovski\u{\i}{} spaces $E_{\varphi }$, where the lower Matuszewska-Orlicz index $\alpha _{\varphi }$ plays the key role. We also give a number of theorems concerning different copies of $l^{\infty }$ in the spaces $E_{\varphi }$ in the natural language of suitable properties of the space $E$ and the function $\varphi$. Our studies are conducted in a full possible generality.
\end{abstract}

\maketitle

\section{Introduction}

The geometry of normed spaces have been intensively developed in the last decades and has a lot of important applications. Moreover, a particular attention has been paid to the study of quasi-normed or $F$-normed spaces, which are important generalizations of normed spaces (see \cites{bhs, chkk, k1, k, KPR84, kmp, kz, Kol2018-Posit, lee, ma04}). Recall also that the so-called $\Delta $-norm is a generalization of both an $F$-norm and a quasi-norm (see \cite{KPR84}). From a different point of view, the theory of modular spaces was and is also being developed (see \cites{kko, ko, mu}).

We will try to link these threads. Namely, we will introduce a new (as far as we know) notion of the functional $\rho $ which we will call a quasi-modular. This is an essential generalization of the concept of convex semi-modular. Furthermore, this quasi-modular $\rho$ is defined in such a way that the respective Minkowski functional of the unit quasi-modular ball gives a quasi-norm. 

Then we consider a special case of quasi-modular spaces, that is, the quasi-normed Calderón--Lozanovski\u{\i}{} spaces $E_{\varphi}$ generated by a quasi-modular $\rho_{\varphi}^{E}$. These spaces have been studied in \cite{kmp} and, in particular case for $E=L^{1}$, in \cite{kz}. Note that, in both these papers, the authors considered only strictly increasing, non-convex Orlicz functions and we admit non-convex, non-decreasing and degenerated Orlicz functions, which gives the full generality of studies. Note also that the topics discussed in \cite{kmp} (and almost all in \cite{kz}) and in our paper are completely different.

First, we will check what are the relationships between a positive Matuszewska-Orlicz index $\alpha_{\varphi}^{E}$ and the fact that $\rho_{\varphi}^{E}$ ($\left\Vert\cdot\right\Vert_{\varphi}$) is a quasi-modular (a quasi-norm), respectively. Then we characterize the spaces $E_{\varphi}$ having order isomorphic copy of $l_{\infty}$ and order linearly isometric copy of $l_{\infty }$. Denote by $P$ a property of having one of such a copy. We will show how the suitable conditions of the Orlicz function $\varphi$ or the properties $P$ of the quasi-Banach ideal space $E$ affect the corresponding property $P$ of the space $E_{\varphi}$. Moreover, we consider also the reverse implication. We will focus on those elements of this puzzle that distinguish the considered quasi-normed case from the well known normed case (which has been intensively and widely studied - see for example \cites{hkm,fh,ma}). We will see that quasi-normed case forces a lot of new techniques and methods in the comparison to the normed case. In particular the conditions $\Delta_{\varepsilon }$ and $\Delta_{2-str}$ of the non-convex Orlicz function $\varphi$ become crucial. Finally, it is worth to mention that we consider both the function and the sequence case which is rare in similar studies of Calderón--Lozanovski\u{\i}{} spaces $E_{\varphi}$ so far, because these cases require different techniques. Of course, considering these two cases in one paper will allow the reader to see the essential differences easily.

\section{Preliminaries}

\begin{definition}
\label{quasi-norm} Given a real vector space $X$ the functional $x\mapsto\Vert x\Vert$ is called a \textit{quasi-norm} if the following three
conditions are satisfied:
\begin{itemize}
\item [$(i)$] $\Vert x\Vert =0$ if and only if $x=0$.
\item [$(ii)$] $\Vert ax\Vert =|a|\Vert x\Vert$ for any $x\in X$ and $a\in\mathbb R$.
\item [$(iii)$] there exists $C=C_{X}\geq1$ such that $\Vert x+y\Vert\leq C(\Vert x\Vert +\Vert y\Vert )$ for all $x,y\in X$.
\end{itemize}
\end{definition}

For $0<p\leq1$, the functional $x\mapsto\Vert x\Vert_{1}$ is called a {\em p-norm} if it satisfies the first two conditions of the quasi-norm and the condition $\Vert x+y\Vert_{1}^{p}\leq \Vert x\Vert_{1}^{p}+\Vert y\Vert_{1}^{p}$ for any $x,y\in X$. Clearly, each p-norm is a quasi-norm. By the \textit{Aoki--Rolewicz theorem} (cf.\ \cite{KPR84}*{Theorem 1.3 on page 7}, \cite{ma04}*{page 86}), given a quasi-norm $\Vert\cdot\Vert$, if $0<p\leq 1$ is such that $C=2^{1/p-1}$, then there exists a $p$-norm $\Vert\cdot\Vert _{1}$ which is equivalent to $\Vert\cdot\Vert$, that is 
\begin{equation}
\Vert x\Vert _{1}\leq \Vert x\Vert \leq 2C\Vert x\Vert _{1}
\label{p-norma}
\end{equation}
for all $x\in X$. The quasi-norm $\Vert \cdot\Vert $ induces a metric topology on $X$: in fact a metric can be defined by $d(x,y)=\Vert x-y\Vert_{1}^{p}$. We say that $X=(X,\Vert\cdot\Vert )$ is a \textit{quasi-Banach space} if it is complete for this metric. Let us note that a lot of important informations and results on quasi-Banach spaces can be found in \cite{k}, see also \cite{KPR84}.

As usual $S(X)$ (resp. $B(X)$) stands for the unit sphere (resp. the closed unit ball) of a real quasi-Banach space $(X,\Vert\cdot\Vert_{X})$.

Recall that a quasi-Banach lattice $E$ is called {\it order continuous} ($E\in({\rm OC})$) if for each sequence $x_{n}\downarrow0$, that is $x_{n}\geq x_{n+1}$ and $\inf_{n}x_{n}=0$, we have $\Vert x_{n}\Vert _{E}\rightarrow 0$ (see \cite{ka},\cite{ltf},\cite{Wnuk}). Moreover, $E\in\left(\rm{OC}\right) $ if and only if for every element $x\in E$ and each sequence $\left( x_{n}\right) $ in $E$ satisfying conditions
$\inf \left\{ x_{n},x_{m}\right\} =0$ for $n\neq m$ and $0\leq x_{n}\leq \left\vert x\right\vert $ we have $\left\Vert x_{n}\right\Vert_{E}\rightarrow 0$ (see Theorem 2.1 in \cite{Kol2018-Posit}).

\section{Quasi-modular spaces}

In this section we introduce the notions of quasi-modular and quasi-modular space.

\begin{definition}\label{quasi-mod} Let $X$ be a real linear space. We say that a function $\rho:X\rightarrow \lbrack 0,\infty]$ is a quasi-modular whenever for all $x,y\in X$ the following conditions are satisfied:
\begin{itemize}
\item [$(i)$] $\rho \left( 0\right) =0$ and the condition $\rho\left(\lambda x\right)\leq1$ for all $\lambda >0$ implies
that $x=0$.
\item [$(ii)$] $\rho \left( -x\right) =\rho \left( x\right)$.
\item[$(iii)$] $\rho(\lambda x)$ is non-decreasing function of $\lambda$, where $\lambda\geq0$.
\item [$(iv)$] There is $M\geq1$ such that 
$$\rho\left(\alpha x+\beta y\right)\leq M\left[\rho\left(x\right)+\rho\left(y\right)\right]$$
provided $\alpha,\beta\geq0$ and $\alpha+\beta=1$.
\item[$(v)$] There is a constant $p>0$ such that for all $\varepsilon>0$ and all $A>0$ there exists  $K=K(\varepsilon,A)\geq1$ such that
$$\rho\left(ax\right)\leq Ka^{p}\rho\left(x\right)+\varepsilon$$
for any $0<a\leq1$ whenever $\rho(x)\leq A$.
\end{itemize}
\end{definition}

\begin{remark}\label{rq-m}{\rm $(i)$ The above definition, especially condition $(v)$, has been introduced in such a way in order both to cover the largest possible class of mappings $\rho$ as well as to provide a quasi-norm (see Theorem \ref{t31}). As we will show in Theorem \ref{r42}, condition $(v)$ can be simplified in some particular cases (more precisely, some particular modulars satisfy condition $(v)$ in simpler or stronger form).

$(ii)$ Obviously a convex modular (more precisely a convex semi-modular) defined in \cite{mu} is in particular a quasi-modular. As we will show in Example \ref{e1} $(ii)$ and $(iii)$, the concepts of quasi-modular and modular (more precisely a semi-modular, which induces an F-norm)  defined in \cite{mu} are incomparable.}
\end{remark}

If $\rho$ is a quasi-modular on $X$, then
\[X_{\rho}:=\left\{x\in X\colon \lim_{\lambda\rightarrow0}\rho(\lambda x)=0\right\}\]
is called a quasi-modular space. It is easy to show that $X_{\rho}$ is linear subspace of $X$. We get also the following

\begin{lemma}\label{l3p} For any quasi-modular $\rho$, we have 
\[X_{\rho}=\left\{x\in X\colon \rho(\lambda x)<\infty\text{ for some }\lambda>0\right\}.\]
\end{lemma}

\begin{proof} Note that in order to prove this lemma, it is enough to show that if $\rho(\lambda_{0}x)<\infty$ for some $\lambda_{0}>0$, then $\lim_{\lambda\rightarrow0}\rho(\lambda x)=0$. Let $x\in X$ and $\rho(\lambda_{0}x)<\infty$ for some $\lambda_{0}>0$. We prove that for any $\varepsilon\in(0,\rho(\lambda_{0}x))$ there exists $\delta=\delta(\varepsilon)>0$ such that $\rho(\lambda x)<\varepsilon$, whenever $\lambda<\delta$. Let $\varepsilon\in(0,\rho(\lambda_{0}x))$ be any fixed and take $K=K(\frac{\varepsilon}{2},\rho(\lambda_{0}x))$ from condition $(v)$ of Definition \ref{quasi-mod}. Then for $\lambda<\delta$, where $\delta=\lambda_{0}(\frac{\varepsilon}{2K\rho(\lambda_{0}x)})^{1/p}$, we get
\[\rho\left(\lambda x\right)=\rho\left(\frac{\lambda}{\lambda_{0}}\lambda_{0}x\right)\leq K\left(\frac{\lambda}{\lambda_{0}}\right)^{p}\rho \left(\lambda_{0}x\right)+\varepsilon/2<\varepsilon.\] 
\end{proof}

\begin{theorem}\label{t31}
Let $\rho $ be a quasi-modular on $X$. Then the functional
\begin{equation*}
\left\Vert x\right\Vert _{\rho }=\inf\left\{\lambda>0:\rho\left(x/\lambda\right)\leq1\right\}
\end{equation*}
is a quasi-norm on $X_{\rho}$.
\end{theorem}

\begin{proof} The condition $\left\Vert 0\right\Vert _{\rho }=0$ is obvious. Suppose $x\neq 0$. Then, by condition $(i)$, there exists $\lambda>0$ such that $\rho\left(\lambda x\right)>1$ and, by condition $(iii)$, $\left\Vert x\right\Vert_{\rho }\geq 1/\lambda$.

For any $x\in X_{\rho}$ and any $\alpha\in \mathbb R$, exactly the same way as in \cite {mu}, we obtain
\[\|\alpha x\|_{\rho}=\inf\left\{\lambda>0:\rho\left(\frac{\alpha x}{\lambda}\right)\leq 1\right\}=|\alpha|\inf\left\{\lambda/|\alpha|>0:\rho\left(\frac{x}{\lambda/\alpha}\right)\leq 1\right\}=|\alpha|\|x\|_{\rho}.\]
Finally, we prove the quasi-triangle inequality. Let $0<\varepsilon<\frac{1}{2M}$ be fixed and take $K=K(\varepsilon,1)$ (see condition $(iv)$ and $(v)$ of Definition \ref{quasi-mod}). Defining
\[C=\left(\frac{K}{\frac{1}{2M}-\varepsilon}\right)^{1/p},\]
for any $x,y\in X_{\rho }$  and any $\delta>0$, we get 
\begin{eqnarray*}
& &\rho\left(\frac{x+y}{C\left( \left\Vert x\right\Vert_{\rho }+\left\Vert y\right\Vert _{\rho }+\delta\right) }\right)\\
&=&\rho\left(\frac{\left\Vert x\right\Vert _{\rho}+\delta/2}{\left\Vert x\right\Vert _{\rho }+\left\Vert y\right\Vert _{\rho}+\delta}\frac{x}{C(\left\Vert x\right\Vert_{\rho}+\delta/2)}+\frac{\left\Vert y\right\Vert_{\rho}+\delta/2}{\left\Vert x\right\Vert_{\rho}+\left\Vert y\right\Vert_{\rho }+\delta}\frac{y}{C(\left\Vert y\right\Vert _{\rho}+\delta/2)}\right)\\
&\leq&M\left[\rho\left(\frac{x}{C(\left\Vert x\right\Vert _{\rho}+\delta/2)}\right)
+\rho\left(\frac{y}{C(\left\Vert y\right\Vert _{\rho}+\delta/2)}\right) \right]\\
&\leq&M\left[\frac{K}{C^{p}}\cdot\rho\left(\frac{x}{(\left\Vert x\right\Vert _{\rho}+\delta/2)}\right)+\varepsilon
+\frac{K}{C^{p}}\cdot\rho\left(\frac{y}{(\left\Vert y\right\Vert _{\rho}+\delta/2)}\right)+\varepsilon \right]\\
&\leq&2M\left(\frac{K}{C^{p}}+\varepsilon\right)=1,
\end{eqnarray*}
whence%
\begin{equation*}
\left\Vert x+y\right\Vert_{\rho }\leq C\left(\left\Vert x\right\Vert_{\rho }+\left\Vert y\right\Vert_{\rho }+\delta\right).
\end{equation*}
By arbitrariness of $\delta$ we have $\left\Vert x+y\right\Vert_{\rho }\leq C(\left\Vert x\right\Vert_{\rho }+\left\Vert y\right\Vert_{\rho})$.
\end{proof}

By the definition of quasi-norm, we get immediately the following

\begin{lemma}\label{le32} Let $\rho $ be a quasi-modular on $X$. Then for any $x\in X_{\rho}$ the following statements hold:

$(i)$ If $\rho(x)\leq1$, then $\|x\|_{\rho}\leq1$.

$(ii)$ If $\rho$ is left continuous $(\lim_{\lambda\rightarrow1-}\rho(\lambda x)=\rho(x)$ for all $x\in X_{\rho})$, then $\rho(x)\leq1$ whenever $\|x\|_{\rho}\leq1$.

$(iii)$ If $\|x\|_{\rho}<1$, then $\rho(x)\leq1$.

$(iv)$ If $\rho$ is right continuous $(\lim_{\lambda\rightarrow1+}\rho(\lambda x)=\rho(x)$ for all $x\in X_{\rho})$, then $\|x\|_{\rho}<1$ whenever $\rho(x)<1$.
\end{lemma}

\begin{remark}\label{re1}{\rm As we will show in Example \ref{e1} $(iv)$ the implication, if $\|x\|_{\rho}<1$ then $\rho(x)<1$, is not always true. Recall, that this implication is true for any modular as well as for any convex modular (see \cite{mu})}.
\end{remark}

\begin{lemma}\label{le33} For any sequence $\left( x_{n}\right) $ in $X_{\rho }$ we have $\lim_{n\rightarrow\infty}\left\Vert x_{n}\right\Vert_{\rho }=0$ if and only if $\lim_{n\rightarrow\infty}\rho\left(\lambda x_{n}\right)=0$ for all $\lambda>0$.
\end{lemma}

\begin{proof} The implication, if $\lim_{n\rightarrow\infty}\rho\left(\lambda x_{n}\right)=0$ for all $\lambda>0$ then $\lim_{n\rightarrow\infty}\left\Vert x_{n}\right\Vert_{\rho }=0$ is obvious. Let now $\|x_{n}\|_{\rho}\rightarrow0$. Fix $\lambda>0$ and $\varepsilon\in(0,1)$ and let $K=K(\varepsilon/2,1)$ be the constant from condition $(v)$ of Definition \ref{quasi-mod}. Then there exists $n_{\lambda,\varepsilon}$ such that $\|\lambda x_{n}\|_{\rho}\leq(\varepsilon/4K)^{1/p}$ for any $n\geq n_{\lambda,\varepsilon}$. Hence for any $a\in((\varepsilon/4K)^{1/p},(\varepsilon/2K)^{1/p})$ we obtain 
\[\rho(\lambda x_{n})=\rho\left(a\frac{\lambda x_{n}}{a}\right)\leq Ka^{p}\rho\left(\frac{\lambda x_{n}}{a}\right)+\frac{\varepsilon}{2}\leq\varepsilon.\]
By arbitrariness of $\varepsilon$, for any $\lambda>0$ we get $\lim_{n\rightarrow\infty}\rho(\lambda x_{n})=0$.
\end{proof}

\section{Quasi-normed Calderón--Lozanowski\u{\i} spaces}

A triple $(T,\Sigma,\mu)$ stands for a positive, complete and $\sigma$-finite measure space and $L^{0}=L^{0}(T,\Sigma,\mu)$ denotes the space of all (equivalence classes of) $\Sigma$-measurable functions $x:T\rightarrow\mathbb R$. For every $x\in L^{0}$ we denote $\supp x=\{ t\in T:x(t)\neq0\}$. Moreover, for any $x,y\in L^{0}$, we write $x\leq y$, if $x(t)\leq y(t)$ almost everywhere with respect to the measure $\mu$ on the set $T$.

A quasi-normed lattice [quasi-Banach lattice] $E=(E,\leq,\Vert\cdot\Vert _{E})$ is called a \textit{quasi-normed ideal space} [\textit{quasi-Banach ideal space} (or a \textit{quasi-Köthe space})] if it is a linear subspace of $L^{0}$ satisfying the following conditions:
\begin{itemize}
\item[$(i)$] If $x\in L^{0}$, $y\in E$ and $|x|\leq |y|$ $\mu$-a.e., then $x\in E $ and $\Vert x\Vert _{E}\leq \Vert y\Vert _{E}$.
\item[$(ii)$] There exists $x\in E$ which is strictly positive on the whole $T$.
\end{itemize}
By $E_{+}$ we denote the positive cone of $E$, that is, $E_{+}={\{x\in E:x\geq 0\}}$. Let $C_{E}$ be the constant from the quasi-triangle inequality for $E$. As usually the symbol $\supp E$ stands a support of $E$ (see \cite{ma}*{p. 169}). Whereas, by $E(w)$ we denote the weighted quasi-normed ideal space, that is,
\[E(w)=\{x\in L^{0}\colon xw\in E\}\]
with the norm $\|x\|_{E(w)}=\|xw\|_{E}$, where $w:T\rightarrow(0,\infty)$ is a measurable weight function.

In the case of quasi-Banach ideal space, we can give equivalent definition of order continuity. Namely, a quasi-Banach ideal space $E$ is called \textit{order continuous} if and only if for every element $x\in E$ and each sequence $\left(x_{n}\right)$ in $E$ such that $0\leq x_{n}\leq\left\vert x\right\vert $ and $x_{n}\rightarrow 0$ $\mu$-a.e.\ we have $\left\Vert x_{n}\right\Vert_{E}\rightarrow0$ (see \cite{Kol2018-Posit}*{Theorem 2.1}). 

An element $x\in E$ is said to be order continuous if for any sequence $(x_{n})$ in $E$ such that $0\leq x_{n}\leq \left\vert x\right\vert $ and $x_{n}\rightarrow 0$ $\mu$-a.e.\ we have $\left\Vert x_{n}\right\Vert_{E}\rightarrow 0$. The subspace $E_{a}$ of all order continuous elements in $E$ is an order ideal in $E$. Obviously, the space $E$ is order continuous if and only if $E_{a}=E$.

We say that quasi-normed ideal space $E$ has the \textit{Fatou property}, if for any $x\in L^{0}$ and any $\left( x_{n}\right) _{n=1}^{\infty }$ in $E_{+}$ such that  $x_{n}\uparrow|x|$ $\mu$-a.e and $\sup_{n\in {N}}\Vert x_{n}\Vert _{E}<\infty$, we get $x\in E$ and $\lim_{n}\Vert x_{n}\Vert_{E}=\Vert x\Vert _{E}$. It is well known, that $E$ has the Fatou property if and only if for any $x\in L^{0}$ and any $\left( x_{n}\right) _{n=1}^{\infty }$ in $E$ such that  $x_{n}\rightarrow x$ $\mu$-a.e and $\liminf_{n\in {N}}\Vert x_{n}\Vert _{E}<\infty$, we have $x\in E$ and $\Vert x\Vert _{E}\leq\liminf_{n}\Vert x_{n}\Vert_{E}$ (cf. \cite{bs}*{Lemma 1.5 on page 4}).

\begin{lemma}\label{l4-1}Let $E$ be quasi-normed ideal space.
\begin{itemize}
\item[$(i)$] If $\lim_{n\rightarrow\infty}\|x-x_{n}\|_{E}=0$, where $x\in E$ and $(x_{n})_{n=1}^{\infty}$ is a sequence in $E$, then $x_{n}\rightarrow x$ locally in measure.
\item[$(ii)$] For any Cauchy sequence $(x_{n})_{n=1}^{\infty}$ in $E$ there exists $x\in L^{0}$ such that $x_{n}\rightarrow x$ locally in measure.
\end{itemize}
\end{lemma}

\begin{proof} This lemma can be proved analogously as Theorem 1 on page 96 in \cite{ka}. Indeed, assuming in (2) on page 96  that $\|x_{n}-x\|_{E}<\frac{\varepsilon}{2^{n}C_{E}^{n}}$, we obtain $\|\chi_{B_{n}}\|_{E}<\frac{1}{2^{n}C_{E}^{n}}$ (see (5) on page 96) and, in consequence, 
\[\|\chi_{D_{n}}\|_{E}\leq\|\chi_{C_{ms_{m}}}\|_{E}\leq\left\|\sum_{k=m+1}^{m+s_{m}}\chi_{B_{k}}\right\|_{E}\leq\sum_{k=m+1}^{m+s_{m}}C_{E}^{(k-m)}\|\chi_{B_{k}}\|_{E}<\sum_{k=m+1}^{m+s_{m}}\frac{C_{E}^{(k-m)}}{2^{k}C_{E}^{k}}<\frac{1}{2^{m}},\]
(see page 97, line 5).\end{proof}

\begin{lemma}\cite{kmp}*{Lemma 2.1}\label{l4-2} A quasi-normed ideal space $E$ with the Fatou property is complete.\end{lemma}

\begin{proof} We recall a short proof of this lemma for the sake of completeness. By the Aoki-Rolewicz theorem, it is enough to show that for any Cauchy sequence $(x_{n})_{n=1}^{\infty}$ in $E$ there exists $x\in E$ such that $\lim_{n\rightarrow\infty}\|x-x_{n}\|_{E}=0$. If $(x_{n})_{n=1}^{\infty}$ is a Cauchy sequence in $E$, then by Lemma \ref{l4-1}$(ii)$, $x_{n}\rightarrow x$ locally in measure for some $x\in L^{0}$. Without loss of generality (passing to subsequences and applying the double extract convergence theorem, if necessary), we can assume $x_{n}\rightarrow x$ $\mu$-a.e. Hence by the Fatou property, we get $x\in E$ and $\|x-x_{n}\|_{E}\leq\liminf_{m\rightarrow\infty}\|x_{m}-x_{n}\|_{E}$ for any $n\in\mathbb N$, which ends the proof.  
\end{proof}

The following basic fact, very well known for Banach ideal spaces (see \cite{ka}*{Lemma 2, p.\ 97}), it also true for quasi-Banach ideal spaces.

\begin{lemma}
Let $\left( E,\left\Vert \cdot \right\Vert _{E}\right) $ be a quasi-Banach ideal space. If $\left\Vert x_{n}\right\Vert _{E}\rightarrow 0,$ then
there exists a subsequence $\left( x_{n_{k}}\right) _{k=1}^{\infty }$, an element $y\in E_{+}$ and a sequence $\varepsilon _{k}\downarrow 0$ such that $\left\vert x_{n_{k}}\right\vert \leq \varepsilon _{k}\cdot y$ for each $k$.
\end{lemma}

\begin{proof}
If $\left\Vert x_{n}\right\Vert_{E}\rightarrow0$, then we can find a subsequence $\left( x_{n_{k}}\right) _{k=1}^{\infty }$ such that $\left\Vert x_{n_{k}}\right\Vert<\frac{1}{C^{k}_{E}2^{k}}$ for any $k\in\mathbb N$. Consequently,
\begin{equation*}
\sum_{k=1}^{\infty}C_{E}^{k}\left\Vert k\cdot x_{n_{k}}\right\Vert\leq\sum_{k=1}^{\infty}C^{k}_{E}\frac{k}{C^{k}_{E}2^{k}}=\sum_{k=1}^{\infty }\frac{k}{2^{k}}<\infty .
\end{equation*}
By Theorem 1.1 from \cite{ma04}, $y:=\sum\nolimits_{k=1}^{\infty}\left\vert k\cdot x_{n_{k}}\right\vert\in E_{+}$. Thus $\left\vert x_{n_{k}}\right\vert\leq\frac{y}{k}$ for all $k$. Taking $\varepsilon_{k}=1/k$ we finish the proof.
\end{proof}

{\bf From then on, we will assume that $E$ is a quasi-Banach ideal space with the Fatou property.} During our studies we will consider three natural classes of  $E$:\newline
$\left( 1\right) $ neither $L_{\infty}\subset E$ nor $E\subset L_{\infty}$,\newline
$\left( 2\right) $ $L_{\infty}\subset E$,\newline
$\left( 3\right) $ $E\subset L_{\infty}$.

Let $T=[0,\gamma)$, $0<\gamma\leq\infty$, and $\mu$ be Lebesgue measure. At the time, the space $E=L_{p}$, $0<p<\infty$, belongs to the class $(1)$ if $\gamma=\infty$ and the class $(2)$ otherwise. Moreover, then the space $L_{1}\cap L_{\infty}$ belongs to the class $(3)$ whenever $\gamma=\infty$. Let now $T=\mathbb{N}$ and $\mu=m$ be a counting measure. Then the space $l_{p}$, $p\in(0,\infty)$, belongs to class $(3)$,  the weighted sequence space $l_{1}(w)$, $w=(w(n))_{n=1}^{\infty}$ and $\sum_{n=1}^{\infty}w(n)<\infty$, belongs to class $(2)$ and Ces\`{a}ro sequence space $ces_{p}$, $1<p<\infty$ (see \cite{kkm} for the respective definition), belongs to class $(1)$.

\begin{remark}\label{r4-0} {\rm Let $E\subset L_{\infty }$. Then, by the closed graph theorem which is still true for quasi-Banach spaces (see \cite{KPR84}*{Theorem. 1.6}), the inclusion is continuous, whence there exists a constant $D_{E}>0$ such that 
\begin{equation}
\left\Vert x\right\Vert_{L_{\infty}}\leq D_{E}\left\Vert x\right\Vert_{E}
\label{wloz}
\end{equation}
for each $x\in E$. Defining
\begin{equation}\label{stala}
a_{E}=\inf\left\{\left\Vert\chi_{A}\right\Vert_{E}:\chi_{A}\in E,\mu\left(A\right)>0\right\},
\end{equation}
by \eqref{wloz}, we have $a_{E}\geq1/D_{E}>0$.
}\end{remark}

\begin{definition}\label{of} A function $\varphi :[0,\infty )\rightarrow[0,\infty]$ is called an \textit{Orlicz function} if  $\varphi$ is non-decreasing, vanishing  and right continuous at $0$, continuous on $(0,b_{\varphi})$, where
\[b_{\varphi }=\sup \left\{ u\geq 0:\varphi \left( u\right)<\infty\right\}\]
and left continuous at $b_{\varphi}$. In the whole paper, excluding Remark \ref{r43} and Example \ref{e1} $(iii)$, we will assume that $\lim_{u\rightarrow\infty}\varphi(u)=\infty$.\end{definition}

Let 
\[a_{\varphi}=\sup\left\{u\geq0:\varphi\left(u\right) =0\right\}.\]
By $\varphi^{-1}$ we denote the generalized inverse of the function $\varphi$ defined by 
\[\varphi^{-1}(v)=\inf\{u\geq0\colon\varphi(u)>v\}\text{ for }v\in[0,\infty)\text{ and }\varphi^{-1}(\infty)=\lim_{v\rightarrow\infty}\varphi^{-1}(v)\] 
(see \cite{n} and \cite{klm}).

The following lemma is an easy exercise (see \cite{klm}*{Lemma 3.1}, for a little different convex case).

\begin{lemma}
\label{skladanie}For any Orlicz function $\varphi$ we have: 
\begin{itemize}
\item[$\left(i\right)$] Let $u\in[0,b_{\varphi})$. Then $\varphi^{-1}(\varphi(u))>u$ if $\varphi$ is constant on the interval $[u,u+\delta) $ for some $\delta>0$ and $\varphi^{-1}(\varphi(u))=u$ otherwise.
\item[$\left(ii\right)$] If $b_{\varphi}<\infty$, then  $\varphi^{-1}(\varphi(b_{\varphi}))=b_{\varphi}$ and $\varphi^{-1}(\varphi(u))=b_{\varphi}<u$ for $b_{\varphi}<u<\infty$.
\item[$\left(iii\right)$] If either $b_{\varphi}=\infty$ or $b_{\varphi}<\infty$ with $\varphi\left(b_{\varphi}\right)=\infty$, then $\varphi(\varphi^{-1}(u))=u$ for any $u\in[0,\infty)$.
\item[$\left(iv\right)$] If $b_{\varphi}<\infty$ and $\varphi\left(b_{\varphi}\right)<\infty$, then $\varphi(\varphi^{-1}(u))=u$ for $u\in[0,\varphi(b_{\varphi})]$ and $\varphi(\varphi^{-1}(u))=\varphi(b_{\varphi})<u$ for $u>\varphi(b_{\varphi})$.
\end{itemize}
From $(i)$ and $(ii)$, in particular, we get:
\begin{itemize}
\item[$\left(v\right)$] $\varphi^{-1}(\varphi(u))=u$ for $a_{\varphi}\leq u<b_{\varphi}$ if either $b_{\varphi}=\infty$ or $b_{\varphi}<\infty$
with $\varphi\left(b_{\varphi}\right)=\infty$ and $\varphi$ is strictly increasing on $[a_{\varphi },b_{\varphi})$.
\item[$\left(vi\right)$] $\varphi^{-1}(\varphi(u))=u$ for $a_{\varphi}\leq u\leq b_{\varphi}$ if $b_{\varphi}<\infty $ and $\varphi\left(b_{\varphi}\right)<\infty$ and $\varphi$ is strictly increasing on  $[a_{\varphi},b_{\varphi}]$. 
\end{itemize}
Finally, note that from $(iii)$--$(iv)$ and $(i)$--$(ii)$ we obtain
\begin{itemize}
\item[$\left(vii\right)$] $\varphi(\varphi^{-1}(u))\leq u$ for all $u\in[0,\infty)$ and $u\leq\varphi^{-1}(\varphi (u))$ if $\varphi (u)<\infty$.
\end{itemize}
\end{lemma}

Recall that for any Orlicz function $\varphi$ the {\it lower Matuszewska--Orlicz index} $\alpha_{\varphi}$ for all arguments is defined by the formula
\begin{eqnarray*}
\alpha^{a}_{\varphi}&=&\sup\{p\in\mathbb R:\text{there exists }K\geq1\text{ such that }\varphi(au)\leq Ka^{p}\varphi(u)\\
& &\text{for any }u\in\mathbb R\text{ and }0<a\leq 1\}.
\end{eqnarray*}
Analogously the  {\it lower Matuszewska--Orlicz indexes} for large and for small arguments are defined as
\begin{eqnarray*}
\alpha^{\infty}_{\varphi}&=&\sup\{p\in\mathbb R:\text{there exist }K\geq1\text{ and }u_{0}>0\text{ such that }\varphi(u_{0})<\infty\text{ and }\\
& &\varphi(au)\leq Ka^{p}\varphi(u)\text{ for any }u\geq u_{0}\text{ and }0<a\leq 1\}
\end{eqnarray*}
and
\begin{eqnarray*}
\alpha^{0}_{\varphi}&=&\sup\{p\in\mathbb R:\text{there exist }K\geq1\text{ and }u_{0}>0\text{ such that }\varphi(au)\leq Ka^{p}\varphi(u)\\
& &\text{for any }u\leq u_{0}\text{ and }0<a\leq 1\},
\end{eqnarray*}
respectively. 

\begin{remark}\label{r40}{\rm $\left( i\right)$ If $0<a_{\varphi }<b_{\varphi }$, then $\alpha _{\varphi}^{0}=\infty $ and we may extend the key inequality in the definition of $\alpha _{\varphi }^{0}$ to any $u_{0}>a_{\varphi}$ such that $\varphi(u_{0})<\infty$. Indeed, let $p>0$.
For each $0\leq u\leq a_{\varphi }$ we have $\varphi\left(au\right)=0=Ka^{p}\varphi\left(u\right)$ for each $K>0$ and $0<a\leq 1$. 
Take any $u_{0}>a_{\varphi}$ such that $\varphi(u_{0})<\infty$. If $a_{\varphi }<u\leq u_{0}$ and $0<a<\frac{a_{\varphi }}{u_{0}}$ we have 
\begin{equation*}
\varphi \left( au\right) \leq \varphi \left( au_{0}\right) \leq \varphi\left( a_{\varphi }\right) =0\leq Ka^{p}\varphi \left( u\right) 
\end{equation*}%
for every $K>0.$ Moreover, 
\begin{equation*}
\sup_{a_{\varphi }<u\leq u_{0}}\sup_{\frac{a_{\varphi }}{u_{0}}\leq a\leq 1}\frac{\varphi \left( au\right) }{a^{p}\varphi \left( u\right) }\leq\sup_{
\frac{a_{\varphi }}{u_{0}}\leq a\leq1}\frac{1}{a^{p}}=\left( \frac{u_{0}}{a_{\varphi }}\right) ^{p}.
\end{equation*}
Thus for $K=\left( \frac{u_{0}}{a_{\varphi }}\right) ^{p}$ we have 
\begin{equation*}
\varphi \left( au\right) \leq Ka^{p}\varphi \left( u\right) 
\end{equation*}%
for all $0<a\leq 1$ and $0\leq u\leq u_{0}.$

$\left( ii\right)$ If $a_{\varphi }=0$ and $\alpha _{\varphi }^{0}>0$ then we may extend the key inequality in the definition of $\alpha _{\varphi }^{0}$ to any $u_{1}$ such that $\varphi \left( u_{1}\right) <\infty .$ Indeed, suppose there is $p>0,$ $u_{0}>0$ and $K>0$ such that 
\begin{equation}
\varphi \left( au\right) \leq Ka^{p}\varphi \left( u\right)   \label{do u0}
\end{equation}
for all $0<a\leq 1$ and $0\leq u\leq u_{0}.$ Take $u_{1}>u_{0}$ satisfying $\varphi \left( u_{1}\right) <\infty .$ Then 
\begin{equation*}
\sup_{u_{0}<u\leq u_{1}}\sup_{\frac{u_{0}}{u_{1}}\leq a\leq 1}\frac{\varphi\left( au\right) }{a^{p}\varphi \left( u\right) }\leq \sup_{\frac{u_{0}}{u_{1}}\leq a\leq 1}\frac{1}{a^{p}}=\left( \frac{u_{1}}{u_{0}}\right) ^{p}.
\end{equation*}
Set $K_{1}=\left( \frac{u_{1}}{u_{0}}\right) ^{p}.$ Now we claim that 
\begin{equation*}
K_{2}:=\sup_{u_{0}<u\leq u_{1}}\sup_{0<a<\frac{u_{0}}{u_{1}}}\frac{\varphi\left(au\right) }{a^{p}\varphi\left(u\right)}<\infty.
\end{equation*}
Otherwise, for each $n\in \mathbb{N}$ we can find $u_{0}<u_{n}\leq u_{1}$ and $0<a_{n}<\frac{u_{0}}{u_{1}}$ such that
\begin{equation*}
\varphi \left( a_{n}u_{n}\right) >na_{n}^{p}\varphi \left( u_{n}\right).
\end{equation*}
Denote $b_{n}:=\frac{a_{n}u_{n}}{u_{0}}.$ Then $b_{n}<1$ and 
\begin{equation*}
\varphi \left( b_{n}u_{0}\right) =\varphi \left( a_{n}u_{n}\right)>na_{n}^{p}\varphi \left( u_{n}\right) =nb_{n}^{p}\left( \frac{u_{0}}{u_{n}}\right) ^{p}\varphi \left( u_{n}\right)\geq nb_{n}^{p}\left(\frac{u_{0}}{u_{1}}\right)^{p}\varphi\left(u_{0}\right).
\end{equation*}
On the other hand, by inequality (\ref{do u0}),
\begin{equation*}
\varphi \left( b_{n}u_{0}\right) \leq Kb_{n}^{p}\varphi \left( u_{0}\right) 
\end{equation*}
which gives a contradiction and proves the claim. Finally, setting $K_{3}=\max \left\{ K,K_{1},K_{2}\right\} $ we conclude that 
\begin{equation*}
\varphi \left( au\right) \leq K_{3}a^{p}\varphi \left( u\right) 
\end{equation*}
for all $0<a\leq 1$ and $0\leq u\leq u_{1}$.

$\left(iii\right)$ If $\alpha_{\varphi}^{\infty }>0$ then we may extend the key inequality in the definition of $\alpha _{\varphi }^{\infty }$ to
any $u_{1}>a_{\varphi }.$ Indeed, suppose there is $p>0,$ $u_{0}>0$ and $K>0$ such that $\varphi(u_{0})<\infty$ and 
\begin{equation*}
\varphi \left( au\right) \leq Ka^{p}\varphi \left( u\right) 
\end{equation*}
for all $0<a\leq 1$ and $u\geq u_{0}$. Take $u_{1}<u_{0}$ satisfying $\varphi\left(u_{1}\right)>0.$ Then
\begin{equation*}
\sup_{u_{1}\leq u<u_{0}}\sup_{0<a\leq 1}\frac{\varphi \left( au\right) }{a^{p}\varphi \left( u\right) }\leq\sup_{0<a\leq1}\frac{\varphi\left(au_{0}\right) }{a^{p}\varphi\left(u_{1}\right) }\leq \sup_{0<a\leq1}\frac{Ka^{p}\varphi \left(u_{0}\right) }{a^{p}\varphi \left( u_{1}\right) }=\frac{%
K\varphi\left(u_{0}\right)}{\varphi\left(u_{1}\right)}.
\end{equation*}
Taking $K_{1}=\max\left\{K,\frac{K\varphi\left(u_{0}\right)}{\varphi\left( u_{1}\right) }\right\} =\frac{K\varphi \left( u_{0}\right) }{\varphi
\left( u_{1}\right) }$ we obtain that 
\begin{equation*}
\varphi \left( au\right) \leq K_{1}a^{p}\varphi \left( u\right) 
\end{equation*}
for all $0<a\leq 1$ and $u\geq u_{1}$.

$\left(iv\right) $ From the above consideration we conclude immediately that if $\alpha _{\varphi }^{\infty }>0$ and $\alpha _{\varphi }^{0}>0$ then $\alpha _{\varphi }^{a}>0$.}
\end{remark}
\vspace{2mm}

\begin{example}\label{ne}{\rm Taking $\varphi_{1}\left(u\right)=\ln\left(1+u\right)$, for $u\geq0$, we easily get
\begin{equation*}
\lim_{u\rightarrow\infty}\frac{\varphi_{1}\left(au\right)}{\varphi_{1}\left(u\right)}=1
\end{equation*}
for any $a\in(0,1)$, whence $\alpha^{\infty}_{\varphi_{1}}=0$. Analogously, defining $\varphi_{2}(0)=0$ and 
\[\varphi_{2}(u)=\frac{1}{\ln\left(1+\frac{1}{u}\right)}\text{\hspace{5mm}for\hspace{5mm}}u>0,\]
we obtain
\begin{equation*}
\lim_{u\rightarrow0^{+}}\frac{\varphi_{2}\left(au\right)}{\varphi_{2}\left(u\right)}=1
\end{equation*}
for any $a\in(0,1)$ and, in consequence, $\alpha^{0}_{\varphi_{2}}=0$. See also Example \ref{de}. 
}\end{example}

For any pair $E$ and $\varphi$ we define $\alpha^{E}_{\varphi}$, by the formula
\begin{equation*}
\alpha^{E}_{\varphi}:=\left\{ 
\begin{array}{ll}
\alpha^{a}_{\varphi}, & \text{when neither }L_{\infty}\subset E\text{ nor }E\subset L_{\infty },\vspace{1mm}\\
\alpha^{\infty}_{\varphi}, & \text{when }L_{\infty}\subset E,\vspace{1mm}\\
\mbox{$\alpha^{0}_{\varphi}$,} & \mbox{when $E\subset L_{\infty}$}.
\end{array}
\right.
\end{equation*}

Given a quasi-Banach ideal space $E$ and an Orlicz function $\varphi$, we define on $L^{0}$ functional $\rho_{\varphi }^{E}$, by 
\begin{equation*}
\rho _{\varphi}^{E}(x):=\left\{ 
\begin{array}{ll}
\Vert \varphi \left( |x|\right) \Vert _{E} & \text{if }\varphi \left(|x|\right) \in E,\text{ }\vspace{1mm}\\ 
\infty & \text{otherwise.}
\end{array}
\right.
\end{equation*}

\begin{theorem}\label{l41} Let $E$ be a quasi-Banach ideal space and $\varphi$ be an Orlicz function. If $\alpha^{E}_{\varphi}>0$, then $\rho_{\varphi }^{E}$ is quasi-modular $($see Definition \ref{quasi-mod}$)$.\end{theorem}

\begin{proof}Obviously $\rho_{\varphi}^{E}(0)=0$ and $\rho_{\varphi}^{E}(-x)=\rho_{\varphi}^{E}(x)$ for any $x\in L^{0}$. Let $x\neq0$. Then there exist $A\in\Sigma$ with $\mu(A)>0$ and $n\in\mathbb N$ such that $\frac{1}{n}\chi_{A}\leq|x|$. If $\varphi(|x|)\notin E$, then $\rho_{\varphi }^{E}(x)=\infty>1$. While, if $\varphi(|x|)\in E$, by $\varphi(\frac{1}{n})\chi_{A}\leq\varphi(|x|)$, we get $\chi_{A}\in E$.  Since $\lim_{u\rightarrow\infty}\varphi(u)=\infty$, we can find $\lambda_{A}>0$ such that $\varphi(\lambda_{A}/n)>1/\|\chi_{A}\|_{E}$ and, in consequence, $\rho_{\varphi }^{E}(\lambda_{A}x)=\|\varphi(\lambda_{A}|x|)\|_{E}\geq\|\varphi(\lambda_{A}/n)\chi_{A}\|_{E}>1$.

For any $x\in L^{0}$ and any $0\leq\lambda_{1}\leq\lambda_{2}$ we have $\varphi(\lambda_{1}|x(t)|)\leq\varphi(\lambda_{2}|x(t)|)$ for $\mu$-a.e.\ $t\in T$, whence $\rho_{\varphi}^{E}(\lambda_{1}x)\leq\rho_{\varphi}^{E}(\lambda_{2}x)$.

Let now $x,y\in L^{0}$ and $\alpha,\beta\geq0$, $\alpha+\beta=1$. Then
\[\varphi(|\alpha x(t)+\beta y(t)|)\leq\varphi(\max(|x(t)|,|y(t)|))\leq\varphi(|x(t)|)+\varphi(|y(t)|)\]
for $\mu$-a.e.\ $t\in T$. Hence, if $\varphi(|x|)\in E$ and $\varphi(|y|)\in E$, we get
\begin{eqnarray*}
\rho_{\varphi}^{E}\left(\alpha x+\beta y\right)&=&\Vert\varphi\left(|\alpha x+\beta y|\right)\Vert _{E}\leq\Vert\varphi\left(|x|\right) +\varphi\left(|y|\right)\Vert_{E}\\
&\leq & C_{E}\left(\Vert\varphi\left(|x|\right)\Vert_{E}+\Vert\varphi\left(|y|\right)\Vert_{E}\right)=C_{E}\left(\rho_{\varphi}^{E}\left(x\right)+\rho _{\varphi}^{E}\left(y\right)\right).
\end{eqnarray*}
Obviously, the inequality $\rho_{\varphi}^{E}\left(\alpha x+\beta y\right)\leq C_{E}\left(\rho_{\varphi}^{E}\left(x\right)+\rho _{\varphi}^{E}\left(y\right)\right)$ holds true, when $\varphi(|x|)\notin E$ or $\varphi(|y|)\notin E$. 

At the end, we will prove that $\rho_{\varphi}^{E}$ satisfies condition $(v)$. Without loss of generality, we can suppose that $0<\rho_{\varphi}^{E}(x)<\infty$ (then in particular we have $a_{\varphi}<b_{\varphi}$). We will consider three case.

If $L_{\infty}\subset E$, then for any $\varepsilon>0$ there exists $u_{1}\in(a_{\varphi},b_{\varphi})$ such that $C_{E}\cdot\varphi(u_{1})\|\chi_{T}\|_{E}<\varepsilon$. Since $\alpha^{E}_{\varphi}=\alpha^{\infty}_{\varphi}>0$, by Remark \ref{r40}, we obtain that there exist $p>0$ and $K=K(\varepsilon)\geq1$ such that $\varphi(au)\leq Ka^{p}\varphi(u)$ for any $u\geq u_{1}$ and $0<a\leq1$. Defining $B=\{t\in T\colon|x(t)|\geq u_{1}\}$, for any $a\in(0,1]$ we get
\begin{eqnarray}\label{i41}
\rho_{\varphi}^{E}(ax)&=&\|\varphi(a|x|)\|_{E}\leq C_{E}\left(\|\varphi(a|x|)\chi_{B}\|_{E}+\|\varphi(a|x|)\chi_{T\backslash B}\|_{E}\right)\\
&\leq & C_{E}Ka^{p}\|\varphi(|x|)\|_{E}+C_{E}\varphi(u_{1})\|\chi_{T}\|_{E}\leq  C_{E}Ka^{p}\rho_{\varphi}^{E}(x)+\varepsilon.\nonumber
\end{eqnarray} 
In the case when neither $L_{\infty}\subset E$ nor $E\subset L_{\infty}$, analogously as above we get 
\begin{equation}\rho_{\varphi}^{E}(ax)=\|\varphi(a|x|)\|_{E}\leq Ka^{p}\|\varphi(|x|)\|_{E}=Ka^{p}\rho_{\varphi}^{E}(x).
\end{equation}
Let now $E\subset L_{\infty}$, take $A>0$ and assume that $\rho_{\varphi}^{E}(x)=\|\varphi(|x|)\|_{E}\leq A$. By the closed graph theorem, there e\-xists constant $D_{E}>0$ such that $\|\varphi(|x|)\|_{L^{\infty}}\leq D_{E}\|\varphi(|x|)\|_{E}\leq AD_{E}$. Hence $\varphi(|x(t)|)\leq\min(AD_{E},\varphi(b_{\varphi}))$ for $\mu$-a.e.\ $t\in T$, whence $|x(t)|\leq u_{2}$ for the same $t$, where $u_{2}=\varphi^{-1}(\min(AD_{E},\varphi(b_{\varphi})))$. Simultaneously, by $\alpha^{E}_{\varphi}=\alpha^{0}_{\varphi}>0$ and Remark \ref{r40}, we obtain that there exist $p>0$ and $K=K(A)\geq1$ such that $\varphi(au)\leq Ka^{p}\varphi(u)$ for any $u\leq u_{2}$ and $0<a\leq1$. Therefore, $\rho_{\varphi}^{E}(ax)\leq Ka^{p}\rho_{\varphi}^{E}(x)$.
\end{proof}

Now we will show that if $\rho_{\varphi}^{E}$ is a quasi-modular (more precisely, if $\rho_{\varphi}^{E}$ satisfies the condition $(v)$ of the definition of the quasi-modular - see Definition \ref{quasi-mod} and also Remark \ref{rq-m}$(i)$), then $\alpha^{E}_{\varphi}>0$.

\begin{theorem}\label{r42}$(i)$ Assume that neither $L_{\infty}\subset E$ nor $E\subset L_{\infty}$. Then $\alpha^{a}_{\varphi}>0$  if and only if there exist constants $p>0$ and  $K\geq1$ such that for any $x\in L^{0}$ and any $0<a\leq1$ we have $\rho_{\varphi}^{E}\left(ax\right)\leq Ka^{p}\rho_{\varphi}^{E}\left(x\right)$.

$(ii)$ Let $L_{\infty}\subset E$. Then $\alpha^{\infty}_{\varphi}>0$ if and only if there is a constant $p>0$ such that for all $v_{0}>0$ there exists $K=K(v_{0})\geq1$ such that for any $x\in L^{0}$ satisfying $|x(t)|\geq v_{0}$ for $\mu$-a.e.\ $t\in T$ and any $0<a\leq1$ we have $\rho_{\varphi}^{E}\left(ax\right)\leq Ka^{p}\rho_{\varphi}^{E}\left(x\right)$.

$(iii)$ Let $E\subset L_{\infty}$. Then $\alpha^{0}_{\varphi}>0$ if and only if there is a constant $p>0$ such that for all $A>0$ there exists  $K=K(A)\geq1$ such that for any $x\in L^{0}$ satisfying $\rho_{\varphi}^{E}(x)\leq A$ and any $0<a\leq1$ we have $\rho_{\varphi}^{E}\left(ax\right)\leq Ka^{p}\rho_{\varphi}^{E}\left(x\right)$.
\end{theorem}

\begin{proof}
The necessity of statements $(i)$ and $(iii)$ follows from the proof of Theorem \ref{l41}. Now we will show the sufficiency of $(iii)$. Let $D\in\Sigma$ be such that $\mu(D)>0$ and $\chi_{D}\in E$. Take $u_{0}>0$ satisfying $0<\varphi(u_{0})<\infty$ and define $A:=\varphi(u_{0})\|\chi_{D}\|_{E}$. Then for any $u\leq u_{0}$ and any $a\in(0,1]$ we get
\begin{eqnarray*}\varphi(au)\|\chi_{D}\|_{E}&=&\|\varphi(au)\chi_{D}\|_{E}=\rho_{\varphi}^{E}\left(au\chi_{D}\right)\\
&\leq& K(A)a^{p}\rho_{\varphi}^{E}\left(u\chi_{D}\right)=K(A)a^{p}\|\varphi(u)\chi_{D}\|_{E}=K(A)a^{p}\varphi(u)\|\chi_{D}\|_{E}
\end{eqnarray*}
So, by arbitrariness of $u$ and $a$, we obtain $\alpha^{0}_{\varphi}>0$. Analogously, we can prove the sufficiency of the condition $\alpha^{a}_{\varphi}>0$ in $(i)$.

Now we will prove statement $(ii)$, that is, let $L_{\infty}\subset E$. First let us notice, that by the proof of Theorem \ref{l41}, we get implication: if $\alpha^{\infty}_{\varphi}>0$, then there is a constant $p>0$ such that for all $\varepsilon>0$ there exists  $K=K(\varepsilon)\geq1$ such that $\rho_{\varphi}^{E}\left(ax\right)\leq Ka^{p}\rho_{\varphi}^{E}\left( x\right)+\varepsilon$ for any $x\in L^{0}$ and any $0<a\leq1$. 

Let $\alpha^{\infty}_{\varphi}>0$ and take $p\in(0,\alpha^{\infty}_{\varphi})$, $v_{0}>0$ and $x\in L^{0}$ such that $\rho_{\varphi}^{E}(x)<\infty$ and $|x(t)|\geq v_{0}$ for $\mu$-a.e.\ $t\in T$. By Remark \ref{r40} (especially $(iii)$ and $(iv)$), there exists $K=K(v_{0})$ such that
\[\varphi(a|x(t)|)\leq Ka^{p}\varphi(|x(t)|)\] for any $a\in(0,1]$ and $\mu$-a.e.\ $t\in T$, whence $\rho_{\varphi}^{E}\left(ax\right)\leq Ka^{p}\rho_{\varphi}^{E}\left(x\right)$.

At the end, we will prove the opposite implication. Take $u_{0}>0$ satisfying $0<\varphi(u_{0})<\infty$ and define $v_{0}=u_{0}$. Then for any $u\geq u_{0}$ and any $a\in(0,1]$ we have
\[\varphi(au)\|\chi_{T}\|_{E}=\rho_{\varphi}^{E}\left(au\chi_{T}\right)\leq K(v_{0})a^{p}\rho_{\varphi}^{E}\left(u\chi_{T}\right)=K(v_{0})a^{p}\varphi(u)\|\chi_{T}\|_{E},\]
and, in consequence, $\alpha^{\infty}_{\varphi}>0$.
\end{proof}

\begin{definition} Let a quasi-Banach ideal space $E$ and an Orlicz function $\varphi$ be such that $\alpha_{\varphi}^{E}>0$. Then Calderón--Loza\-nov\-ski{\u{\i}} space $E_{\varphi}$ is defined by 
\begin{equation*}
E_{\varphi}=\{x\in L^{0}:\lim_{\lambda \rightarrow 0}\rho_{\varphi}^{E}(\lambda x)=0\}.
\end{equation*}
\end{definition}

By Theorem \ref{l41} and Lemma \ref{l3p}, $E_{\varphi}$ is quasi-modular space and
\begin{equation*}
E_{\varphi}=\{x\in L^{0}:\rho_{\varphi }^{E}(\lambda x)<\infty\text{ for some }\lambda>0\}.
\end{equation*}
 Moreover, by Theorem \ref{t31}, the functional
\begin{equation*}
\Vert x\Vert_{\varphi }=\inf\left\{\lambda>0:\rho_{\varphi}^{E}\left(x/\lambda\right)\leq1\right\} ,
\end{equation*}
is a quasi-norm, called {\em Luxemburg-Nakano quasi-norm}. It is easy to show, that $E_{\varphi}=(E_{\varphi},\leq,\|\cdot\|_{\varphi})$ is quasi-normed ideal space. We get also the following

\begin{lemma}\label{f_x} For any quasi-Banach ideal space $E$ and any Orlicz function $\varphi$ the fo\-llo\-wing assertions hold:
\begin{itemize}
\item[$(i)$] For any $x\in E_{\varphi }$ the function $f_{x}\left(\alpha\right):=\rho_{\varphi }^{E}(\alpha x)$ for $\alpha>0$ is non-decreasing and left conti\-nu\-o\-us.
\item[$(ii)$] For any $x\in E_{\varphi }$ we have $\|x\|_{\varphi}\leq1$ if and only if $\rho_{\varphi}^{E}(x)\leq1$.
\item[$(iii)$] The quasi-normed ideal space $E_{\varphi}$ has the Fatou property and, in consequence, $E_{\varphi}$ is complete.
\end{itemize}
\end{lemma}

\begin{proof} The assertion $(i)$ follows immediately from the properties of $\varphi$ and $E$ (recall that $E$ has the Fatou property). Next, by $(i)$ and Lemma \ref{le32}, we obtain $(ii)$. Proceeding analogously as in \cite{fh}*{Theorem 12} (see also \cite{kmp}*{Lemma 2.2$(ii)$}), we get that $E_{\varphi}$ has the Fatou property. Hence, by Lemma \ref{l4-2}, $E_{\varphi}$ is complete.
\end{proof}

\begin{remark}{\rm $(i)$ If $E=L^{1}$ ($E=l^{1}$), then $E_{\varphi }$ is the Orlicz function (sequence) space $L^{\varphi}$ ($l^{\varphi}$) equipped with the Luxemburg-Nakano quasi-norm (see \cite{kz}). If $E$ is a Lorentz function (sequence) space $\Lambda_{1,w}$ ($\lambda_{1,w}$) (see \cite{kama} for the respective definition), then $E_{\varphi }$ is the corresponding Orlicz-Lorentz function (sequence) space $\Lambda_{\varphi,w}$ ($\lambda_{\varphi,w}$), equipped with the Luxemburg-Nakano quasi-norm. 

On the other hand, if $\varphi(u)=u^{p}$, $1\leq p<\infty$ [$0<p<1$] then $E_{\varphi }$ is the $p$-convexification [concavification] $E^{(p)}$ of $E$ with the quasi-norm $\Vert x\Vert _{E^{(p)}}=\Vert |x|^{p}\Vert _{E}^{1/p}$. If $\varphi (u)=0$ for $0\leq u\leq 1$ and $\varphi (u)=\infty $ for $u>1$, then $E_{\varphi }=L^{\infty }$ with equality of the norms.

Finally, note that our results can be easily applied to quasi-normed Orlicz spaces and quasi-normed Orlicz-Lorentz spaces. 

$(ii)$ It is known that there is a connection between the space $E_{\varphi}$ and the Calderón-Lozanov\-skii interpolation construction $\psi\left( E,L_{\infty}\right)$ for normed-Banach ideal spaces $E_{\varphi}$ ($\varphi $ is a convex Orlicz function) and $\psi\left( E,L_{\infty }\right)$ ($\psi $ is a homogeneous, concave function on $\mathbb{R}_{+}^{2}$) - see \cite{ma}*{Example 2, p.\  178}. However, there is a similar relation if $\varphi $ is a non-convex Orlicz function and $\psi $ is positively homogeneous, non-decreasing with respect to each variable. Moreover, such a function $\psi $ leads to quasi-Banach lattices $\psi\left(E_{0},E_{1}\right)$ which have application in the interpolation theory (see \cite{rt}).
}\end{remark}

\begin{remark}\label{r42-5} {\rm Let $E\subset L_{\infty}$. Applying inequality \eqref{wloz} we conclude that for every $x\in E_{\varphi }$ such that $\rho_{\varphi}^{E}(x)=\left\Vert\varphi\left(\left\vert x\right\vert\right)\right\Vert_{E}\leq1$ we have
\[\left\vert x\left(t\right)\right\vert\leq\varphi^{-1}\left(D_{E}\right)\]
for $\mu $-a.e.\ $t\in T$ (note that if $b_{\varphi}<\infty$ and $\varphi(b_{\varphi})<D_{E}$, then $\varphi^{-1}(D_{E})=\varphi^{-1}(\varphi(b_{\varphi}))=b_{\varphi}$). We claim that for each $x\in E_{\varphi}$ such that $\rho_{\varphi}^{E}(x)=\left\Vert\varphi\left(\left\vert x\right\vert\right)\right\Vert_{E}\leq 1$ we get
\[\left\vert x\left( t\right) \right\vert\leq\varphi^{-1}\left(1/a_{E}\right)\]
for $\mu $-a.e.\ $t\in T$ (where $a_{E}$ is defined by formula \eqref{stala}), which gives a more optimal inequality than above. Moreover, the number $a_{E}>0$ in several classical ideal spaces (for example in Orlicz spaces) can be easily calculated. Although the proof of this claim is obvious and standard we present the details for reader's convenience.

Let $x\in E_{\varphi}$ be such that $\rho_{\varphi}^{E}(x)=\left\Vert \varphi\left(\left\vert x\right\vert\right)\right\Vert_{E}\leq1$. Suppose $\left\vert x\left(t\right)\right\vert>\varphi^{-1}(1/a_{E})$ for $\mu$-a.e.\ $t\in A$ with $\mu\left(A\right)>0$. By $\rho_{\varphi}^{E}(x)=\left\Vert \varphi\left(\left\vert x\right\vert\right)\right\Vert_{E}\leq1$, we have $\varphi^{-1}(1/a_{E})<\left\vert x\left(t\right)\right\vert\leq b_{\varphi}$, if $\varphi(b_{\varphi}) <\infty$ and $\varphi^{-1}(1/a_{E})<\left\vert x\left(t\right)\right\vert<b_{\varphi}$, if $\varphi(b_{\varphi}) =\infty$ for $\mu$-a.e.\ $t\in A$. By the definition of $\varphi^{-1}$, we have $\varphi\left(\left\vert x\left(t\right)\right\vert\right)>1/a_{E}$ for $\mu$-a.e.\ $t\in A$. Then we can find a number $n\in\mathbb{N}$ and a set $A_{n}\in\Sigma$ with $\mu\left(A_{n}\right)>0$ such that $\varphi\left( \left\vert x\left(t\right)\right\vert\right)>1/a_{E}+1/n$ for $\mu $-a.e.\ $t\in A_{n}$. Hence
\begin{equation*}
\left\Vert\varphi\left(\left\vert x\right\vert\right)\right\Vert_{E}\geq\left\Vert\varphi\left(\left\vert x\right\vert\right)\chi_{A_{n}}\right\Vert _{E}\geq\left(1/a_{E}+1/n\right)\left\Vert\chi_{A_{n}}\right\Vert _{E}\geq\left( 1/a_{E}+1/n\right)a_{E}>1,
\end{equation*}
which gives a contradiction.

In particular, if $1/a_{E}\leq\varphi(b_{\varphi})$, then for any $x\in E_{\varphi}$ and any $\lambda>0$ such that $\rho_{\varphi}^{E}(\lambda x)=\left\Vert\varphi\left(\left\vert\lambda x\right\vert\right)\right\Vert_{E}\leq1$, we have $|\lambda x(t)|\leq\varphi^{-1}(1/a_{E})$ for $\mu$-a.e.\ $t\in T$ and, in consequence, $(E_{\varphi},\|\cdot\|_{\varphi})\equiv(E_{\psi},\|\cdot\|_{\psi})$ for any Orlicz function $\psi$ such that $\psi(u)=\varphi(u)$ for every $u\in[0,\varphi^{-1}(1/a_{E})]$.}\end{remark}

\begin{remark}\label{r43}{\rm Now we will show the naturalness of the assumption $\lim_{u\rightarrow\infty}\varphi(u)=\infty$ in the definition of Orlicz function (see Definition \ref{of}). Obviously, if $\alpha^{a}_{\varphi}>0$ or $\alpha^{\infty}_{\varphi}>0$, then $\lim_{u\rightarrow\infty}\varphi(u)=\infty$. Simultaneously, for $\varphi(u)=\min(u^{2},1)$, we have $\alpha^{0}_{\varphi}>0$ and  $\lim_{u\rightarrow\infty}\varphi(u)=1$. Let $E\subset L_{\infty}$, $\lim_{u\rightarrow\infty}\varphi(u)<\infty$ and $\alpha^{E}_{\varphi}=\alpha^{0}_{\varphi}>0$. Then condition (i) of Definition \ref{quasi-mod} holds whenever $\lim_{u\rightarrow\infty}\varphi(u)>1/a_{E}$, where $a_{E}$ is defined by formula \eqref{stala}, and, in consequence, $\rho_{\varphi}^{E}$ is quasi-modular. Moreover, defining then new Orlicz function $\psi$, by $\psi(u)=\varphi(u)$ for $u\in[0,\varphi^{-1}(1/a_{E})]$ and $\psi(u)=u-(\varphi^{-1}(1/a_{E})-1/a_{E})$ for $u>\varphi^{-1}(1/a_{E})$, we get $\lim_{u\rightarrow\infty}\psi(u)=\infty$ and $(E_{\varphi},\|\cdot\|_{\varphi})\equiv(E_{\psi},\|\cdot\|_{\psi})$ (see Remark \ref{r42-5}).}
\end{remark}

\begin{example}\label{e1}{\rm Assume $T=[0,\infty)$ and $\mu$ is Lebesgue measure.

$(i)$ If $E=L_{1}$ and $\varphi(u)=u^{p}$ for $u\geq0$, $p>0$, then $\rho_{\varphi}^{E}$ is a quasi-modular (a convex modular for $p>1$) as well as a modular (see \cite{mu}). We have
\[\|x\|_{\varphi}=\|x\|_{p}=\left(\int_{0}^{\infty}|x(t)|^{p}dt\right)^{\frac{1}{p}}.\]
Simultaneously, the F-norm is given by
\[|||x|||_{\varphi}:=\inf\left\{\lambda>0\colon\rho_{\varphi}^{E}\left(x/\lambda\right)\leq\lambda\right\}=\left(\int_{0}^{\infty}|x(t)|^{p}dt\right)^{\frac{1}{1+p}}\]
(see \cite{mu}). Note that $\|\cdot\|_{\varphi}$ and $|||\cdot|||_{\varphi}$ are not equivalent.

$(ii)$ Let $E=L_{(1/4)}$ and $\varphi(u)=u^{2}$ for $u\geq0$. Obviously $\rho_{\varphi}^{E}$ is a quasi-modular. Simultaneously, for $x=\chi_{[0,1)}$ and $y=\chi_{[1,2)}$ we get 
\[\rho_{\varphi}^{E}\left(\frac{1}{2}x+\frac{1}{2}y\right)=4>2=\rho_{\varphi}^{E}(x)+\rho_{\varphi}^{E}(y).\]
Thus $\rho_{\varphi}^{E}$ is not a modular.

$(iii)$ If $E=L_{1}$ and $\varphi(u)=\arctan(u)$ for $u\geq0$ (see Definition \ref{of}), then $\rho_{\varphi}^{E}$ is a modular (see \cite{mu}). Now, we will show that $\rho_{\varphi}^{E}$ is not a quasi-modular, more precisely, $\rho_{\varphi}^{E}$ does not satisfy condition $(v)$ of Definition \ref{quasi-mod}. Let $p>0$ and take $A=\pi/2$ and $\varepsilon=\pi/8$. Then for $x_{n}=n\chi_{[0,1)}$ and $a_{n}=1/n$, we obtain $\rho_{\varphi}^{E}(x_{n})\leq\pi/2$, $\rho_{\varphi}^{E}(a_{n}x_{n})=\pi/4$ and $\lim_{n\rightarrow\infty}(a_{n})^{p}=0$. In consequence, for any $K\geq1$ there exists $n\in\mathbb N$ such that $\rho_{\varphi}^{E}(a_{n}x_{n})>Ka_{n}^{p}\rho_{\varphi}^{E}(x_{n})+\pi/8$.

$(iv)$ Let $E=L_{1}$, $\varphi(u)=u$ for $u\in[0,1]$ and $\varphi(u)=\max(1,u-1)$ for $u>1$. For $x=\chi_{[0,1)}$ we have $\rho_{\varphi}^{E}(x)=\rho_{\varphi}^{E}(2x)=1$. Simultaneously, $\rho_{\varphi}^{E}(x/\lambda)>1$ for $\lambda<1/2$, so $\|x\|_{\varphi}=\frac{1}{2}$ (see Remark \ref{re1}).
}\end{example}
\vspace{3mm}

Recall the notion of uniform monotonicity (see for example \cites{b,hkm1,lee}) which plays important role in the theory of Banach lattices. A quasi-Banach lattice $\left(E,\left\Vert\cdot\right\Vert_{E}\right)$ is said to be {\em uniformly monotone} provided for each $\varepsilon>0$ there exists $\delta=\delta\left(\varepsilon\right)>0$ such that for all $x,y\in E_{+}$ with $\left\Vert x\right\Vert_{E}=1$ we have $\left\Vert x+y\right\Vert _{E}\geq1+\delta$ whenever $\left\Vert y\right\Vert_{E}\geq\varepsilon$.

\begin{lemma}\label{jum}
If $E$ is uniformly monotone, then for each $\varepsilon_{1}>0$ and $A>0$ there exists $\delta_{1}=\delta_{1}\left(\varepsilon_{1},A\right)
=\delta\left(\frac{\varepsilon_{1}}{A}\right)>0$ $($here the function $\delta\left(\cdot\right) $ comes from the definition of the uniform monotonicity$)$ such that for all $x,y\in E_{+}$ we have $\left\Vert x+y\right\Vert_{E}\geq\left\Vert x\right\Vert_{E}\left(1+\delta_{1}\right)$ whenever $\left\Vert y\right\Vert_{E}\geq\varepsilon_{1}$ and $\left\Vert x\right\Vert_{E}\leq A$.
\end{lemma}

\begin{proof} The proof can be found in \cite{hkm1} (the same for the quasi-norm). However, we will need the precise form of the function $\delta_{1}\left(\varepsilon_{1},A\right)$, so we present the proof for reader's convenience.

Let $\varepsilon_{1}>0$ and $A>0$. Take $x,y\in E_{+}$ such that $\left\Vert y\right\Vert_{E}\geq\varepsilon_{1}$ and $\left\Vert x\right\Vert_{E}\leq A$. Denote 
\begin{equation*}
\widetilde{x}=\frac{x}{\left\Vert x\right\Vert_{E}}\text{ and }\widetilde{y}=\frac{y}{\left\Vert x\right\Vert_{E}}.
\end{equation*}
Then $\left\Vert\widetilde{x}\right\Vert_{E}=1$ and $\left\Vert\widetilde{y}\right\Vert_{E}\geq\frac{\varepsilon_{1}}{A}$. By uniform monotonicity of $E$, there exists $\delta_{1}=\delta_{1}\left(\varepsilon_{1},A\right)=\delta\left(\frac{\varepsilon_{1}}{A}\right)>0$ such that 
$\left\Vert\widetilde{x}+\widetilde{y}\right\Vert_{E}\geq1+\delta_{1}$, whence
\begin{equation*}
\left\Vert x+y\right\Vert_{E}\geq\left\Vert x\right\Vert_{E}\left(1+\delta_{1}\right).
\end{equation*}\end{proof}

Recall that the assumption $\alpha _{\varphi }^{E}>0$ is important to show that $\rho_{\varphi }^{E}$ is a quasi-modular (see Theorem \ref{l41} and also Theorem \ref{r42}) and, consequently, that $\Vert\cdot\Vert_{\varphi}$ is a quasi-norm. However, if we know nothing about the index $\alpha_{\varphi }^{E}$ then still we may define a functional 
\begin{equation*}
\rho _{\varphi }^{E}(x):=\left\{\begin{array}{cc}
\Vert\varphi\left(|x|\right)\Vert_{E} & \text{if }\varphi\left(|x|\right)\in E,\text{ } \\ 
\infty & \text{otherwise,}
\end{array}\right. 
\end{equation*}
and the set 
\begin{equation*}
E_{\varphi}=\{x\in L^{0}:\lim_{\lambda\rightarrow0}\rho_{\varphi}^{E}(\lambda x)=0\}.
\end{equation*}
Then, it is easy to see that $E_{\varphi }$ is a linear space and we may consider the functional 
\begin{equation*}
\Vert x\Vert _{\varphi}=\inf\left\{\lambda>0:\rho_{\varphi}^{E}\left(x/\lambda\right)\leq 1\right\} ,\text{ for }x\in E_{\varphi},
\end{equation*}
which satisfies the conditions $\left(i\right) $ and $\left( ii\right) $ of the quasi-norm definition. We are going to show that, under some natural assumptions, the condition $\alpha _{\varphi }^{E}>0$ can be even necessary, that is to say, if the functional $\Vert\cdot\Vert_{\varphi}$ is a quasi-norm, then $\alpha_{\varphi}^{E}>0$. Since the below result is only some illustration how natural is the assumption that $\alpha _{\varphi}^{E}>0$, we will limit ourselves only to one case of the ideal space $E$.

\begin{theorem}\label{tkol}
Suppose $\varphi$ is a finitely valued, strictly increasing Orlicz function. Let $\left(E,\left\Vert\cdot\right\Vert_{E}\right)$ be a $p$-normed
ideal space over non-atomic measure space $(T,\Sigma,\mu )$ such that $E$ is uniformly monotone. Assume that neither $L_{\infty }\subset E$ nor $E\subset L_{\infty}$. If $\left(E_{\varphi},\Vert\cdot\Vert_{\varphi}\right)$ is quasi-normed space, then $\alpha_{\varphi}^{a}>0$.
\end{theorem}

\begin{proof}
Denote by $C\geq1$ the constant from the quasi-triangle inequality for $\left(E_{\varphi},\Vert\cdot\Vert_{\varphi}\right)$. Let $\delta_{0}=\delta\left(1/2\right)$ be the constant from the definition of the uniform monotonicity. Fix $s>0$. Recall also that if $E$ is uniformly monotone, then $E$ is order continuous (see \cite{lee}*{Proposition 2.4}). Thus, by the proof of Theorem 2.4 in \cite{kmp}, the function $\nu$, defined by $\nu(A)=\|\chi_{A}\|_{E}$ for all $A\in\Sigma$, $\chi_{A}\in E$, is the submeasure in the sense of \cite{Dobr}*{Definition 1}, whence by \cite{Dobr}*{Theorem 10}, $\nu$ has the Darboux property. In consequence, we can find a set $A\in\Sigma$, $\chi_{A}\in E$ satisfying 
\begin{equation*}
\varphi\left(s\right)=\frac{1}{\left\Vert\chi_{A}\right\Vert_{E}\left(1+\delta _{0}\right)}.
\end{equation*}
Take a set $B\in\Sigma$ of positive measure such that $\chi _{B}\in E$, $A\cap B=\emptyset$ and $\left\Vert\chi_{B}\right\Vert_{E}=\frac{1}{2}\left\Vert\chi_{A}\right\Vert_{E}$. Applying Lemma \ref{jum} we conclude that 
\begin{equation}
\left\Vert\chi_{A\cup B}\right\Vert_{E}\geq\left\Vert\chi_{A}\right\Vert_{E}\left(1+\delta_{0}\right). \label{um1}
\end{equation}
It is well known that 
\begin{equation*}
\left\Vert\chi_{A}\right\Vert_{\varphi}=\frac{1}{\varphi^{-1}\left(\frac{1}{\left\Vert\chi_{A}\right\Vert_{E}}\right)},
\end{equation*}
where $\varphi^{-1}$ is the general right-inverse to $\varphi$. Indeed, applying Lemma \ref{skladanie}, we conclude that $\left\Vert\varphi \left(\frac{\chi_{A}}{\lambda}\right)\right\Vert_{E}\leq1$ if and only if $\frac{1}{\lambda}\leq\varphi^{-1}\left(\frac{1}{\left\Vert\chi_{A}\right\Vert_{E}}\right)$. In consequence,
\begin{equation*}
C\geq\frac{\left\Vert\chi_{A}+\chi_{B}\right\Vert_{\varphi}}{\left\Vert\chi_{A}\right\Vert_{\varphi}+\left\Vert\chi_{B}\right\Vert_{\varphi}}\geq\frac{\left\Vert\chi_{A}+\chi_{B}\right\Vert_{\varphi}}{2\left\Vert\chi_{A}\right\Vert_{\varphi}}=\frac{\left\Vert\chi_{A\cup B}\right\Vert _{\varphi}}{2\left\Vert\chi_{A}\right\Vert_{\varphi}}=\frac{\varphi^{-1}\left(\frac{1}{\left\Vert\chi_{A}\right\Vert_{E}}\right)}{2\varphi^{-1}\left(\frac{1}{\left\Vert\chi_{A\cup B}\right\Vert_{E}}\right)}.
\end{equation*}
Moreover, applying also (\ref{um1}), we get
\begin{equation*}
2C\varphi^{-1}\left(\frac{1}{\left\Vert\chi_{A}\right\Vert_{E}\left(1+\delta_{0}\right)}\right)\geq2C\varphi^{-1}\left(\frac{1}{\left\Vert\chi_{A\cup B}\right\Vert_{E}}\right)\geq\varphi^{-1}\left(\frac{1}{\left\Vert\chi_{A}\right\Vert_{E}}\right)
\end{equation*}
and, applying again Lemma \ref{skladanie}, we obtain 
\begin{equation*}
\varphi\left[2C\varphi^{-1}\left(\frac{1}{\left\Vert\chi_{A}\right\Vert_{E}\left(1+\delta_{0}\right)}\right)\right]\geq\varphi\left[\varphi^{-1}\left(\frac{1}{\left\Vert\chi_{A}\right\Vert_{E}}\right)\right]=\frac{1}{\left\Vert\chi_{A}\right\Vert _{E}}
\end{equation*}
Taking $C_{1}=2C$, see Lemma \ref{skladanie}, we have
\begin{equation*}
\left(1+\delta_{0}\right)\varphi\left(s\right)\leq\varphi\left[2C\varphi^{-1}\left(\varphi\left(s\right)\right)\right]=\varphi\left[
2Cs\right]=\varphi\left[C_{1}s\right] .
\end{equation*}
for each $s>0$. For every $a\geq1$ there is $m\in\mathbb{N}$ such that $C_{1}^{m-1}\leq a<C_{1}^{m}$. Fix $p>0$ satisfying $p=\frac{\ln \left(1+\delta_{0}\right)}{\ln C_{1}}$. Then $\left(1+\delta _{0}\right)^{m-1}=\left(C_{1}^{m-1}\right)^{p}$ and 
\begin{equation*}
\varphi\left(as\right)\geq\varphi\left(C_{1}^{m-1}s\right)\geq\left(1+\delta_{0}\right)^{m-1}\varphi\left(s\right)=\left(C_{1}^{m-1}\right)^{p}\varphi \left( s\right)\geq\left(\frac{a}{C_{1}}\right)^{p}\varphi\left(s\right)=a^{p}C_{1}^{-p}\varphi\left(s\right)
\end{equation*}
for each $s>0$. Setting $u:=as$ and $b:=1/a$ we conclude that
\begin{equation*}
\varphi\left(bu\right)\leq b^{p}C_{1}^{p}\varphi\left(u\right)
\end{equation*}
for any $u>0$ and each $b\in(0,1]$. It means that $\alpha_{\varphi}^{a}>0$.
\end{proof}

From Theorems \ref{t31}, \ref{l41} and \ref{tkol}, taking $E=L_{1}([0,\infty))$, we get immediately the respective charac\-te\-ri\-zation for Orlicz spaces proved directly\ in \cite{kz}*{Theorem 1.8}.

\section{Order continuity and copies of $l^{\infty}$}

\begin{remark}{\rm It is easy to show, that if $a_{\varphi}=b_{\varphi}$, then $(E_{\varphi},\|\cdot\|_{\varphi})\equiv\left(L_{\infty},\frac{1}{b_{\varphi}}\|\cdot\|_{\infty}\right)$. Therefore, from now on, we will assume that $a_{\varphi}<b_{\varphi}$.}
\end{remark}

Recall that an Orlicz function $\varphi$ satisfies the condition $\Delta_{2}$ for all $u\in\mathbb R_{+}$ ($\varphi\in\Delta_{2}(\mathbb R_{+})$ for short) if there exists a constant $K>0$ such that the inequality 
\begin{equation}
\varphi(2u)\leq K\varphi(u)\label{roo1}
\end{equation}
holds for any $u\in\mathbb R_{+}$ (then we have $a_{\varphi}=0$ and $b_{\varphi}=\infty$). Analogously, we say that an Orlicz function $\varphi$ satisfies the condition $\Delta_{2}$ at infinity [at zero] ($\varphi\in\Delta_{2}({\infty})$ [$\varphi\in\Delta_{2}({0})$] for short) if there exist constants $K,u_{0}\in(0,\infty)$ such that $\varphi(u_{0})<\infty$ [$\varphi(u_{0})>0$] and inequality (\ref{roo1}) holds for any $u\geq u_{0}$ [$0\leq u\leq u_{0}$], respectively. Clearly, if $\varphi\in\Delta_{2}({\infty})$ [$\varphi\in\Delta_{2}({0})$], then $b_{\varphi}=\infty$ [$a_{\varphi}=0$].

For any quasi-Banach ideal space $E$ and any Orlicz function $\varphi$ we say that $\varphi$ satisfies condition $\Delta^{E}_{2}$ ($\varphi\in\Delta^{E}_{2}$ for short) if:\\
$(1)$ $\varphi\in\Delta_{2}(\mathbb R_{+})$ whenever neither $L_{\infty}\subset E$ nor $E\subset L_{\infty}$,\\
$(2)$ $\varphi\in\Delta_{2}(\infty)$ whenever $L_{\infty}\subset E$,\\
$(3)$ $\varphi\in\Delta_{2}(0)$ whenever $E\subset L_{\infty}$.

\begin{theorem}\label{tko1}
Suppose $\mu$ is nonatomic and $E_{a}\neq \left\{ 0\right\}$. If $\varphi\notin\Delta_{2}\left(\infty\right)$, then the space $E_{\varphi}$
contains an order linearly isometric copy of $l_{\infty}$.
\end{theorem}

\begin{proof} First assume that $b_{\varphi }=\infty$. Since $E_{a}\neq\left\{0\right\}$, there is a set $A\in\Sigma$ with $0<\mu \left(A\right) <\infty $ and $\chi_{A}\in E_{a}$. Let $\left( A_{n}\right) _{n=1}^{\infty }$ be a sequence of measurable, pairwise disjoint subsets of $A$. Applying the assumption that $\varphi\notin\Delta_{2}\left(\infty\right) $ we can find a number $u_{1}>0$ such that
\begin{equation*}
\varphi \left(u_{1}\right)\left\Vert\chi_{A_{1}}\right\Vert_{E,p}\geq \frac{1}{8C_{E}^{3}}
\end{equation*}
and
\begin{equation*}
\varphi\left(2u_{1}\right)>8C_{E}^{3}\cdot\varphi\left(u_{1}\right),
\end{equation*}
where $\left\Vert\cdot\right\Vert_{E,p}$ is a $p$-norm equivalent to $\left\Vert \cdot \right\Vert _{E}$ (see inequality \eqref{p-norma}). Denote by $\Sigma _{1}$ the $\sigma$-algebra of all measurable subsets of $A$. Note that the function $\nu\left(B\right)=\left\Vert\chi_{B}\right\Vert_{E,p}$ defined for all $B\in \Sigma _{1}$ is the submeasure on $\Sigma _{1}$ in the sense of Definition 1 from \cite{Dobr} which is shown in the proof of Theorem 2.4 in \cite{kmp}. Consequently, by Theorem 10 in \cite{Dobr}, $\nu$ has the\ Darboux property. Thus we find a set $B_{1}\in \Sigma_{1},B_{1}\subset A_{1}$ such that
\begin{equation*}
\varphi\left(u_{1}\right)\left\Vert\chi _{B_{1}}\right\Vert_{E,p}=\frac{1}{8C_{E}^{3}}.
\end{equation*}%
By inequality \eqref{p-norma} we get
\begin{equation*}
\frac{1}{8C_{E}^{3}}\leq \varphi \left( u_{1}\right) \left\Vert\chi_{B_{1}}\right\Vert _{E}\leq \frac{1}{4C_{E}^{2}}.
\end{equation*}%
Applying the assumption that $\varphi\notin \Delta_{2}\left(\infty\right) $ we can find a number $u_{2}>u_{1}$ satisfying
\begin{equation*}
\varphi\left(u_{2}\right)\left\Vert\chi_{A_{2}}\right\Vert_{E,p}\geq\frac{1}{16C_{E}^{4}}
\end{equation*}
and
\begin{equation*}
\varphi\left(\left(1+\frac{1}{2}\right)u_{2}\right)>16C_{E}^{4}\cdot\varphi\left(u_{2}\right) .
\end{equation*}
Analogously as above, there is set $B_{2}\in \Sigma _{1},B_{2}\subset A_{2}$ such that
\begin{equation*}
\varphi \left( u_{2}\right) \left\Vert \chi _{B_{2}}\right\Vert _{E,p}=\frac{1}{16C_{E}^{4}}.
\end{equation*}
In consequence,
\begin{equation*}
\frac{1}{16C_{E}^{4}}\leq\varphi\left(u_{2}\right)\left\Vert\chi_{B_{2}}\right\Vert _{E}\leq\frac{1}{8C_{E}^{3}}.
\end{equation*}
Proceeding in such a way by induction we can find an increasing sequence $\left(u_{n}\right) $ of positive numbers and a sequence of $\left(
B_{n}\right) $ of sets in $\Sigma _{1}$ such that $B_{n}\subset A_{n}$ for any $n\in \mathbb{N}$ and
\begin{equation}
\varphi\left(\left(1+\frac{1}{n}\right)u_{n}\right)>2^{n+2}C_{E}^{n+2}\cdot\varphi\left(u_{n}\right) \label{kl1}
\end{equation}
and
\begin{equation}
\frac{1}{2^{n+2}C_{E}^{n+2}}\leq\varphi\left(u_{n}\right)\left\Vert\chi_{B_{n}}\right\Vert _{E}\leq\frac{1}{2^{n+1}C_{E}^{n+1}}. \label{kl2}
\end{equation}
Let $x_{n}=u_{n}\chi _{B_{n}}$ for $n\in \mathbb{N}$ and
\begin{equation*}
x=\sum_{n=1}^{\infty }u_{n}\chi _{B_{n}}.
\end{equation*}
Applying inequality (\ref{kl2}) we get
\begin{equation*}
C_{E}\sum_{n=1}^{\infty}C_{E}^{n}\left\Vert\varphi\left( u_{n}\right)\chi_{B_{n}}\right\Vert_{E}\leq\sum_{n=1}^{\infty }\frac{C_{E}^{n+1}}{2^{n+1}C_{E}^{n+1}}\leq\frac{1}{2}.
\end{equation*}
Thus, by Theorem 1.1 from \cite{ma04}, we conclude that $\varphi \left(x\right)=\sum_{n=1}^{\infty}\varphi\left( u_{n}\right)\chi_{B_{n}}\in E$
and
\begin{equation*}
\rho_{\varphi}^{E}\left(x\right)=\left\Vert\varphi\left(x\right)\right\Vert_{E}=\left\Vert\sum_{n=1}^{\infty}\varphi\left(u_{n}\right)\chi _{B_{n}}\right\Vert_{E}\leq C_{E}\sum_{n=1}^{\infty}C_{E}^{n}\left\Vert\varphi\left(u_{n}\right)\chi _{B_{n}}\right\Vert_{E}\leq\frac{1}{2}.
\end{equation*}
Moreover, for each $\lambda>1$ there exists $n_{\lambda }\in\mathbb{N}$ such that $\lambda >1+\frac{1}{n_{\lambda}},$ whence, by inequalities (\ref{kl1}) and (\ref{kl2}),
\begin{eqnarray*}
\rho_{\varphi }^{E}\left(\lambda x\right)&=&\left\Vert\varphi\left(\lambda x\right)\right\Vert_{E}\geq\left\Vert\varphi\left(\left(1+
\frac{1}{n_{\lambda }}\right)x\right)\right\Vert_{E}\geq\left\Vert\varphi\left(\left(1+\frac{1}{n_{\lambda }}\right) u_{n_{\lambda}}\right)\chi_ {B_{n_{\lambda }}}\right\Vert_{E} \\
&>&2^{n_{\lambda }+2}C_{E}^{n_{\lambda }+2}\varphi \left( u_{n_{\lambda}}\right) \left\Vert \chi _{B_{n_{\lambda }}}\right\Vert _{E}\geq1.
\end{eqnarray*}
Thus $\left\Vert x\right\Vert _{\varphi }=1.$ Define
\begin{eqnarray}\label{yyy}
y_{1} &=&x_{1}+x_{3}+x_{5}+... \\ 
y_{2} &=&x_{2}+x_{6}+x_{10}+...\nonumber
\end{eqnarray}
and, by the induction, we define the element $y_{m}$ ($m\in \mathbb{N}$, $m\geq 3$) to be the sum of every second term $x_{n}$ of 
\begin{equation}\label{yyyy}
\sum_{n=1}^{\infty }x_{n}-\sum_{i=1}^{m-1}y_{i}
\end{equation}
starting from the first term of the rest. We can prove the same way as above that $\rho _{\varphi }^{E}\left( y_{m}\right) \leq 2^{-m}$ and $\left\Vert y_{m}\right\Vert_{\varphi}=1$ ($m\in\mathbb{N}$). Finally, we define an operator $P:l_{\infty }\rightarrow E_{\varphi }$ by the formula
\begin{equation*}
P\left(z\right)=\sum_{m=1}^{\infty }z_{m}y_{m},
\end{equation*}
where $z=\left(z_{m}\right)\in l_{\infty }$. Clearly, $P$ is linear and $P\left( z\right) \geq 0$ whenever $z\geq 0.$ Moreover, for each $z\in
l_{\infty }\backslash\left\{0\right\} $ we have
\begin{equation*}
\rho_{\varphi }^{E}\left(\frac{P\left(z\right)}{\left\Vert z\right\Vert_{l_{\infty}}}\right)=\left\Vert\sum_{m=1}^{\infty}\varphi\left(\frac{\left\vert z_{m}\right\vert y_{m}}{\left\Vert z\right\Vert _{l_{\infty }}}\right) \right\Vert_{E}\leq\left\Vert\sum_{m=1}^{\infty}\varphi\left(y_{m}\right) \right\Vert_{E}\leq\left\Vert\varphi\left(x\right)\right\Vert_{E}\leq \frac{1}{2},
\end{equation*}
whence $\left\Vert P\left( z\right) \right\Vert _{\varphi }\leq \left\Vert z\right\Vert _{l_{\infty }}.$ Furthermore, for each $\lambda<1$ we find $m_{\lambda }\in \mathbb{N}$ such that $\frac{\left\vert z_{m_{\lambda}}\right\vert}{\lambda\left\Vert z\right\Vert_{l_{\infty}}}=\lambda
_{0}>1$. In consequence,
\begin{equation*}
\left\Vert\frac{P\left(z\right)}{\lambda\left\Vert z\right\Vert_{l_{\infty}}}\right\Vert _{\varphi}\geq\left\Vert\frac{\left\vert z_{m_{\lambda}}\right\vert}{\lambda\left\Vert z\right\Vert_{l_{\infty }}}y_{m_{\lambda}}\right\Vert_{\varphi }=\frac{\left\vert z_{m_{\lambda}}\right\vert }{\lambda\left\Vert z\right\Vert_{l_{\infty}}}>1.
\end{equation*}
Thus $\left\Vert P\left( z\right)\right\Vert_{\varphi }>\lambda\left\Vert z\right\Vert_{l_{\infty}}$. Since $\lambda<1$ was arbitrary, $\left\Vert P\left( z\right)\right\Vert _{\varphi}\geq\left\Vert z\right\Vert_{l_{\infty }}$. Finally, $\left\Vert P\left(z\right)\right\Vert_{\varphi}=\left\Vert z\right\Vert _{l_{\infty }}$ for each $z\in l_{\infty }$ which finishes this part of the proof.

At the end assume that $b_{\varphi}<\infty$. Let
\begin{equation*}
u_{n}=\frac{2n+1}{2n+2}\cdot b_{\varphi}
\end{equation*}
and the sequence $\left(A_{n}\right)_{n=1}^{\infty}$ be defined as above. If $\varphi\left(u_{n}\right)\left\Vert\chi_{A_{n}}\right\Vert_{E,p}\leq \frac{1}{2^{n+2}C_{E}^{n+2}}$, then we set $B_{n}=A_{n}$, whence $\varphi\left(u_{n}\right)\left\Vert\chi_{B_{n}}\right\Vert_{E}\leq\frac{1}{2^{n+1}C_{E}^{n+1}}$ (see inequality (\ref{kl2})). In the opposite case (that is when $\varphi\left(u_{n}\right)\left\Vert\chi_{A_{n}}\right\Vert_{E,p}>\frac{1}{2^{n+2}C_{E}^{n+2}}$), analogously as above, we conclude that there are sets $B_{n}\in \Sigma _{1}$ such that $B_{n}\subset A_{n}$ and $\varphi\left(u_{n}\right)\left\Vert\chi_{B_{n}}\right\Vert_{E,p}=\frac{1}{2^{n+2}C_{E}^{n+2}}$ for any $n\in\mathbb{N}$. Hence
\begin{equation*}
\varphi\left(u_{n}\right)\left\Vert\chi_{B_{n}}\right\Vert_{E}\leq\frac{1}{2^{n+1}C_{E}^{n+1}}
\end{equation*}
for the same $n$. Again, taking $x_{n}=u_{n}\chi _{B_{n}}$ for $n\in \mathbb{N}$ and
\begin{equation*}
x=\sum_{n=1}^{\infty}u_{n}\chi _{B_{n}}
\end{equation*}%
we get $\rho_{\varphi}^{E}\left(x\right)\leq\frac{1}{2}$. Moreover, for any fixed $\lambda>1$  there exists $n_{\lambda}\in\mathbb{N}$ such that $\lambda>1+\frac{1}{n_{\lambda}}$. Since $\varphi\left(\left(1+\frac{1}{n_{\lambda}}\right)u_{n_{\lambda}}\right)=\infty$, we have $\varphi\left(\lambda x\right)\notin E$, whence $\rho_{\varphi }^{E}\left(\lambda x\right)=\infty$. By arbitrariness of $\lambda>1$, we get $\left\Vert x\right\Vert_{\varphi}=1.$ Taking the sequence $\left( y_{m}\right) $ as in the first part of the proof and defining  the operator $P:l_{\infty }\rightarrow E_{\varphi }$ by the formula
\begin{equation*}
P\left(z\right)=\sum_{m=1}^{\infty}z_{m}y_{m},
\end{equation*}%
where $z=\left(z_{m}\right)\in l_{\infty }$, we conclude that $P$ is an order linear isometry.
\end{proof}

\begin{theorem}\label{tko2}
Assume $\mu$ is nonatomic, $L_{\infty}\not\subset E$ and $\supp(E_{a})=T$. If $\varphi\notin\Delta_{2}\left(0\right)$, then the space $E_{\varphi}$ contains an order linearly isometric copy of $l_{\infty}$.
\end{theorem}

\begin{proof} First assume that $a_{\varphi}=0$. Since $\varphi\notin\Delta_{2}\left(0\right)$, there exists a decreasing to zero sequence $(u_{n})_{n=1}^{\infty}$ such that $u_{1}<b_{\varphi}/2$ and 
\[\varphi\left(\left(1+\frac{1}{n}\right)u_{n}\right)>2^{n+2}C_{E}^{n+2}\varphi(u_{n})\]
for any $n\in\mathbb N$. Moreover, since $\supp(E_{a})=T$, there exists a sequence $(A_{n})_{n=1}^{\infty}$ of measurable subset of $T$ such that $\mu(A_{n})<\infty$, $A_{n}\subset A_{n+1}$ and $\chi_{A_{n}}\in E_{a}$ for any $n\in\mathbb N$ and $\bigcup_{n=1}^{\infty}A_{n}=T$ (see \cite{ma}*{p. 169--170}). Since $E$ has the Fatou property and $L_{\infty}\not\subset E$, we get $\lim_{n\rightarrow\infty}\|\chi_{A_{n}}\|_{E}=\infty$. So, there is $n_{1}$ such that 
\[\varphi(u_{1})\|\chi_{A_{n_{1}}}\|_{E,p}\geq\frac{1}{2^{3}C_{E}^{3}},\]
where $\|\cdot\|_{E,p}$ is a p-norm equivalent to $\|\cdot\|_{E}$ (see inequality \eqref{p-norma}). Let $\Sigma_{1}$ be the $\sigma$-algebra of all measurable subsets of $A_{n_{1}}$. Analogously as in the proof of Theorem \ref{tko1}, we conclude that $\nu\left(B\right)=\left\Vert\chi_{B}\right\Vert_{E,p}$ defined for all $B\in \Sigma _{1}$ is the submeasure on $\Sigma _{1}$ and $\nu$ has the Darboux property. Therefore, we can find a set $B_{1}\in \Sigma_{1},B_{1}\subset A_{n_{1}}$ such that
\[\varphi(u_{1})\|\chi_{B_{1}}\|_{E,p}=\frac{1}{2^{3}C_{E}^{3}}.\]
By inequality \eqref{p-norma}, we obtain
\[\frac{1}{2^{3}C_{E}^{3}}\leq\varphi(u_{1})\|\chi_{B_{1}}\|_{E}\leq\frac{1}{2^{2}C_{E}^{2}}.\]
Let $A_{n}^{2}=A_{n}\backslash A_{n_{1}}$ for $n>n_{1}$. Applying again the Fatou property, we get $\lim_{n\rightarrow\infty}\|\chi_{A_{n}^{2}}\|_{E}=\infty$. Thus, there is $n_{2}>n_{1}$ such that
\[\varphi(u_{2})\|\chi_{A_{n_{2}}^{2}}\|_{E,p}\geq\frac{1}{2^{4}C_{E}^{4}},\]
Proceeding analogously as above, we conclude that
\[\frac{1}{2^{4}C_{E}^{4}}=\varphi(u_{2})\|\chi_{B_{2}}\|_{E,p}\leq\varphi(u_{2})\|\chi_{B_{2}}\|_{E}\leq2C\varphi(u_{2})\|\chi_{B_{2}}\|_{E,p}=\frac{1}{2^{3}C_{E}^{3}}.\]
for some $B_{2}\subset A_{n_{2}}^{2}$. Proceeding in such a way by induction, we can find a sequence $(B_{n})_{n=1}^{\infty}$ of pairwise disjoint measurable sets such that 
\[\frac{1}{2^{n+2}C_{E}^{n+2}}\leq\varphi(u_{n})\|\chi_{B_{n}}\|_{E}\leq\frac{1}{2^{n+1}C_{E}^{n+1}}\]
for any $n\in\mathbb N$. Defining $x_{n}=u_{n}\chi_{B_{n}}$ for $n\in\mathbb N$, $x=\sum_{n=1}^{\infty}u_{n}\chi_{B_{n}}$ and the sequence $(y_{m})_{m=1}^{\infty}$ just like in the proof of Theorem \ref{tko1} (see formulas \eqref{yyy} and \eqref{yyyy}), we get $\rho_{\varphi}^{E}(x)\leq\frac{1}{2}$, $\|x\|_{\varphi}=1$, $\rho_{\varphi}^{E}(y_{m})\leq2^{-m}$ and $\|y_{m}\|_{\varphi}=1$ for $m\in\mathbb N$. In consequence, the operator $P:l_{\infty }\rightarrow E_{\varphi }$, where $P\left(z\right)=\sum_{m=1}^{\infty}z_{m}y_{m}$ for $z=\left(z_{m}\right)\in l_{\infty }$, is an order linear isometry.

Finally, suppose $a_{\varphi}>0$. Let $u_{n}=a_{\varphi}$ for any $n\in\mathbb N$. Proceeding analogously as in first part of this proof, we can find a sequence $(B_{n})_{n=1}^{\infty}$ of pairwise disjoint measurable sets such that
\begin{equation}\label{jjr}1<\varphi\left(\left(1+\frac{1}{n}\right)u_{n}\right)\|\chi_{B_{n}}\|_{E}
\end{equation}
for any $n\in\mathbb N$ (if $b_{\varphi}\leq2a_{\varphi}$, then there exists $n_{0}\in\mathbb N$ such that $(1+1/n_{0})a_{\varphi}<b_{\varphi}$ and in inequality \eqref{jjr} we substitute $(1+1/n_{0})$ in place of $(1+1/n)$ for $n\leq n_{0}$). Defining $x$ and $y_{m}$, $m\in\mathbb N$, as above, we get $\rho_{\varphi}^{E}(x)=\rho_{\varphi}^{E}(y_{m})=0$, $m\in\mathbb N$, and $\|x\|_{\varphi}=\|y_{m}\|_{\varphi}=1$, $m\in\mathbb N$, whence we obtain that the operator $P$ defined as above is an order linear isometry.
\end{proof}

Theorems \ref{tko1} and \ref{tko2}  are generalizations Theorems 1 and 2 from \cite{hkm} (normed case of $E_{\varphi}$) as well as Proposition 1.11 from \cite{kz} (quasi-normed case of $E_{\varphi}$ for $E=L_{1}$ and strictly increasing $\varphi$). From Theorems \ref{tko1} and \ref{tko2}, we get the following

\begin{corollary}\label{c1} $(i)$ Let $\mu$ be nonatomic, $L_{\infty}\subset E$ and $E_{a}\neq\{0\}$. If $\varphi\notin\Delta_{2}^{E}$, then the space $E_{\varphi}$ contains an order linearly isometric copy of $l_{\infty}$.

$(ii)$ Let $\mu$ be nonatomic. Assume that neither $L_{\infty}\subset E$ nor $E\subset L_{\infty}$ and $\supp(E_{a})=T$. If $\varphi\notin\Delta_{2}^{E}$, then the space $E_{\varphi}$ contains an order linearly isometric copy of $l_{\infty}$.
\end{corollary}

\begin{remark}{\rm
Let $\mu$ be nonatomic and $E\subset L_{\infty}$. By Remark \ref{r4-0}, we have $E_{a}=\{0\}$. However, in some cases, proceeding analogously as in the proof of Theorem \ref{tko2} one can show that if $\varphi\notin\Delta_{2}^{E}$ ($\varphi\notin\Delta_{2}(0)$), then the space $E_{\varphi}$ contains an order linearly isometric copy of $l_{\infty}$ (see example below).
}\end{remark}

\begin{example}\label{olo}{\rm
Let $T=[0,\infty)$, $\mu$ be Lebesgue measure, $\psi(u)=u^{2}$ for $u\in[0,1]$ and $\psi(u)=\infty$ for $u>1$. Then the Orlicz space
\[L_{\psi}=\left\{x\in L^{0}\colon\int_{0}^{\infty}\psi(\lambda|x(t)|)dt<\infty\text{ for some }\lambda>0\right\}\]
is contained in $L_{\infty}$. Obviously, for the both norms, that is, for the Luxemburg-Nakano norm $\|\cdot\|_{\psi}$ as well as for the Orlicz-Amemiya norm 
\[\|x\|_{\psi}^{O}:=\inf_{k>0}\frac{1}{k}\left\{1+\int_{0}^{\infty}\psi(k|x(t)|)dt\right\}\]
(see \cite{hm}), we have $(L_{\psi})_{a}=\{0\}$. Simultaneously, for any measurable set $A\subset[0,\infty)$ we get $\|\chi_{A}\|_{\psi}=\max(1,\sqrt{\mu(A)})$ and $\|\chi_{A}\|_{\psi}^{O}=1+\mu(A)$ if $\mu(A)\leq1$ and $\|\chi_{A}\|_{\psi}^{O}=2\sqrt{\mu(A)}$ otherwise. So, for $(E,\|\cdot\|_{E})=(L_{\psi},\|\cdot\|_{\psi})$ or $(E,\|\cdot\|_{E})=(L_{\psi},\|\cdot\|_{\psi}^{O})$, proceeding analogously as in the proof of Theorem \ref{tko2}, we obtain that $E_{\varphi}$ contains an order linearly isometric copy of $l_{\infty}$ whenever $\varphi\notin\Delta_{2}(0)$.
}\end{example}

\begin{theorem}\label{toc} Let $\mu$ be nonatomic. Then $E_{\varphi}$ is order continuous if and only if $E$ is order continuous and $\varphi$ satisfies the condition $\Delta_{2}^{E}$.
\end{theorem}

\begin{proof} {\em Sufficiency.} Since $E$ is order continuous, we have $L_{\infty}\subsetneq E$ or (neither $L_{\infty}\subset E$ nor $E\subset L_{\infty}$). Take any $x\in E_{\varphi}$ and a sequence $(x_{n})$ in $(E_{\varphi})_{+}$ such that $0\leftarrow x_{n}\leq|x|$. By the definition of $E_{\varphi}$, there exists $\lambda_{x}>0$ such that $\varphi(\lambda_{x}|x|)\in E$. So, for any fixed $\lambda>0$ we can find $i=i(\lambda)\in\mathbb N$ such that $\lambda\leq2^{i}\lambda_{x}$. In consequence, we get
\[\varphi(\lambda|x|)\leq\varphi(2^{i}\lambda_{x}|x|)\leq K^{i}\varphi(\lambda_{x}|x|)+\varphi(2^{i}\lambda_{x}u_{0})\chi_{T}\in E,\]
whenever $L_{\infty}\subsetneq E$, where $K$ and $u_{0}$ are the constants from the condition $\Delta_{2}^{E}$ $(\Delta_{2}(\infty))$, and
\[\varphi(\lambda|x|)\leq\varphi(2^{i}\lambda_{x}|x|)\leq K^{i}\varphi(\lambda_{x}|x|)\in E,\]
otherwise. Therefore, for any $\lambda>0$, we have $0\leftarrow\varphi(\lambda x_{n})\leq\varphi(\lambda|x|)\in E$. By order continuity of $E$, we obtain $\lim_{n\rightarrow\infty}\rho_{\varphi}^{E}(\lambda x_{n})=\lim_{n\rightarrow\infty}\|\varphi(\lambda x_{n})\|_{E}=0$ for the same $\lambda$. Finally, by Lemma \ref{le33}, we get $\lim_{n\rightarrow\infty}\|x_{n}\|_{\varphi}=0$.

{\em Necessity.} We have $b_{\varphi}=\infty$. Indeed, if $b_{\varphi}<\infty$, then for any $A\in\Sigma$, $\chi_{A}\in E$, we get $\|\chi_{A}\|_{\varphi}\geq1/b_{\varphi}$ and, in consequence, $E_{\varphi}$ is not order continuous.

Assume now that $E$ is not order continuous. Then there exist $x\in E$ and a sequence $(x_{n})$ in $E_{+}$ such that $0\leftarrow x_{n}\leq|x|$ for any $n\in\mathbb N$ and $\|x_{n}\|_{E}\nrightarrow0$. Without loss of generality, we can assume that $\|x_{n}\|_{E}>1$ for any $n\in\mathbb N$. Defining $y=\varphi^{-1}(|x|)$ and $y_{n}=\varphi^{-1}(x_{n})$ for any $n\in\mathbb N$, we have $0\leftarrow y_{n}\leq y$ and $\rho_{\varphi}^{E}(y_{n})=\|\varphi(y_{n})\|_{E}=\|x_{n}\|_{E}>1$ for the same $n$. By the definition of $\|\cdot\|_{\varphi}$, we get $\|y_{n}\|_{\varphi}>1$ for any $n\in\mathbb N$ and, in consequence,  $E_{\varphi}$ is not order continuous again.

At the end, let $E$ be order continuous. Then $L_{\infty}\subsetneq E$ or (neither $L_{\infty}\subset E$ nor $E\subset L_{\infty}$), whence by Corollary \ref{c1}, $\varphi$ satisfies the condition $\Delta_{2}^{E}$.
\end{proof}

By Theorem \ref{toc} and Proposition 2.2 from \cite{lee} we get immediately the following

\begin{corollary}\label{cko1} Let $\mu$ be nonatomic. Then $E_{\varphi}$ contains an order isomorphic copy of $l_{\infty}$ if and only if $E$ contains an order isomorphic copy of $l_{\infty}$ or $\varphi\notin\Delta_{2}^{E}$.\end{corollary}

\begin{theorem}\label{tko3}
Suppose $E$ is a quasi-Banach ideal space over the counting measure space $\left(\mathbb{N},2^{\mathbb{N}},m\right) $ such that $l_{\infty
}\not\subset E$.

$\left(i\right)$ If $a_{\varphi}>0$, then $E_{\varphi}$ contains an order linearly isometric copy of $l_{\infty}$.

$\left( ii\right) $ Assume that $a_{\varphi}=0$ and there exist a number $d>0$ and an increasing sequence of natural numbers $\left(i_{j}\right)_{j=1}^{\infty }$ such that $\left\Vert e\left(i_{j}\right)\right\Vert_{E}\leq d$ for any $j\in \mathbb{N}$ and $\sum\limits_{j=1}^{\infty}e\left(i_{j}\right)\notin E$, where $e\left(i\right)$ is the $i$-th unit vector. If $\varphi \notin \Delta _{2}\left( 0\right) ,$ then $E_{\varphi}$
contains an order linearly isometric copy of $l_{\infty }.$
\end{theorem}

\begin{remark}{\rm 
Note that if $l_{\infty}\not\subset E$ and $\sup_{i\in\mathbb N}\left\Vert e\left(i\right)\right\Vert_{E}<\infty$, then E satisfies the assumption of the statement $(ii)$ of Theorem \ref{tko3}. Therefore, the assumption of the statement $(ii)$ of Theorem \ref{tko3} is satisfied by all symmetric sequence spaces $E$ with $E\neq l_{\infty }$ as well as by Ces\`{a}ro sequence space $ces_{p}$, $1<p<\infty$ (see \cite{kkm} for the respective definition). On the other hand, this assumption is satisfied also by the weighted sequence space $l_{p}\left(w\right)$, $p\in(0,1]$, where $w\left(2i-1\right)=i$ and $w\left(2i\right)=\frac{1}{i}$ although $\sup_{i\in\mathbb N}\left\Vert e\left(i\right)\right\Vert_{E}=\infty$.}\end{remark}

\begin{proof}
$\left(i\right)$ Suppose $2a_{\varphi}<b_{\varphi}$. Since $l_{\infty}\not\subset E$ and $E$ has the Fatou property, we conclude that
\begin{equation*}
\lim_{n\rightarrow\infty}\left\Vert\sum\limits_{i=k}^{n}e\left(i\right)\right\Vert_{E}=\infty
\end{equation*}
for each $k\in\mathbb{N}$. Consequently, by induction, we can find an increasing sequence of natural numbers $\left(k_{n}\right)_{n=1}^{\infty}$ such that 
\begin{equation}
\varphi\left(\left(1+\frac{1}{n}\right)a_{\varphi}\right)\left\Vert\sum\limits_{i=k_{n-1}+1}^{k_{n}}e\left(i\right)\right\Vert_{E}>1,
\label{pow_1}
\end{equation}
for each $n\in \mathbb{N}$, where $k_{0}=0$. If $2a_{\varphi }\geq b_{\varphi },$ then we find a number $n_{1}$ with $\left(1+\frac{1}{n}\right)a_{\varphi}<b_{\varphi}$ for all $n\geq n_{1}$ (recall that $a_{\varphi}<b_{\varphi}$) and we can work above with the sequence of numbers $n\geq n_{1}$. Defining
\begin{equation*}
x_{n}=\sum\limits_{i=k_{n-1}+1}^{k_{n}}a_{\varphi}\cdot e\left(i\right)\text{ and }x=\sum\limits_{n=1}^{\infty}x_{n}
\end{equation*}
we have $\rho_{\varphi}^{E}(x_{n})=0$ for each $n\in\mathbb{N}$ and $\rho_{\varphi }^{E}(x)=0$. Furthermore, for each $\lambda>1$ there is a number $n_{\lambda}\in\mathbb{N}$ such that $\lambda>1+\frac{1}{n_{\lambda}}$, whence by (\ref{pow_1}),
\begin{equation*}
\rho_{\varphi }^{E}(\lambda x)\geq\left\Vert\varphi\left(\left(1+\frac{1}{n_{\lambda}}\right)x\right)\right\Vert_{E}\geq\left\Vert\varphi
\left(\left(1+\frac{1}{n_{\lambda}}\right)x_{n_{\lambda}}\right)\right\Vert_{E}>1.
\end{equation*}
Thus $\left\Vert x\right\Vert_{\varphi}=1$. Defining the sequence $\left(y_{m}\right)_{m=1}^{\infty}$ like in the proof of Theorem \ref{tko1}
(see formulas \eqref{yyy} and \eqref{yyyy}), we obtain that $\rho_{\varphi}^{E}(y_{m})=0$ and $\left\Vert y_{m}\right\Vert_{\varphi }=1$ for each $m\in \mathbb{N}$. In consequence, the operator $P:l_{\infty }\rightarrow E_{\varphi}$, defining by $P(z)=\sum\limits_{m=1}^{\infty}z_{m}y_{m}$ for $z=\left(z_{m}\right)\in l_{\infty}$, is an order linear isometry.

$\left(ii\right)$ Since $\varphi\notin\Delta_{2}\left(0\right)$, we find  a decreasing to zero sequence 
$\left( u_{n}\right) _{n=1}^{\infty }$ such that $u_{1}<\frac{1}{2}b_{\varphi }$ and
\begin{equation}
\varphi\left(u_{n}\right)\leq\frac{1}{dC_{E}^{n+2}2^{n+2}}\hspace{2mm}\text{ and }\hspace{2mm}\varphi\left(\left(1+\frac{1}{n}\right)u_{n}\right)\geq C_{E}^{n+2}2^{n+2}\varphi\left(u_{n}\right).\label{d2}
\end{equation}
Moreover, applying the Fatou property of $E$ and the assumption of the statement (ii) we conclude that
\begin{equation*}
\lim_{n\rightarrow\infty}\left\Vert \sum\limits_{j=k}^{n}e\left(i_{j}\right) \right\Vert_{E}=\infty
\end{equation*}
for each $k\in\mathbb{N}$. Let $k_{1}$ be the smallest natural number for
which 
\begin{equation*}
\varphi\left(2u_{1}\right)\left\Vert\sum\limits_{j=1}^{k_{1}}e\left(i_{j}\right)\right\Vert_{E}>1.
\end{equation*}
Denoting $N_{1}=\left\{1,2,\ldots,k_{1}\right\}$ and setting $\Vert\sum\limits_{j\in\emptyset}e\left( i_{j}\right)\Vert _{E}=0$ we get
\begin{equation*}
\varphi\left(2u_{1}\right)\left\Vert\sum\limits_{j\in N_{1}\backslash\left\{k_{1}\right\}}e\left(i_{j}\right)\right\Vert_{E}\leq1.
\end{equation*}
Let now $k_{2}$ be the smallest natural number for which 
\begin{equation*}
\varphi\left(\left(1+\frac{1}{2}\right)u_{2}\right)\left\Vert\sum\limits_{j=k_{1}+1}^{k_{2}}e\left( i_{j}\right)\right\Vert_{E}>1.
\end{equation*}
Denoting $N_{2}=\left\{ k_{1}+1,\dots,k_{2}\right\} $ we get
\begin{equation*}
\varphi\left(\left(1+\frac{1}{2}\right)u_{2}\right)\left\Vert\sum\limits_{j\in N_{2}\backslash\left\{k_{2}\right\}}e\left(i_{j}\right)\right\Vert_{E}\leq1.
\end{equation*}
Following similarly by the induction we can find an increasing sequence of natural numbers $\left(k_{n}\right)_{n=1}^{\infty}$ such that the sets $N_{n}=\left\{k_{n-1}+1,\dots,k_{n}\right\}$ for $n\in\mathbb{N}$, where $k_{0}=0$, are pairwise disjoint and 
\begin{equation}
\varphi\left(\left(1+\frac{1}{n}\right)u_{n}\right)\left\Vert\sum\limits_{j\in N_{n}}e\left(i_{j}\right)\right\Vert_{E}>1\text{ and }\varphi\left(\left( 1+\frac{1}{n}\right)u_{n}\right)\left\Vert\sum\limits_{j\in N_{n}\backslash\left\{k_{n}\right\} }e\left(i_{j}\right)\right\Vert_{E}\leq1.  \label{pow1b}
\end{equation}
Defining
\begin{equation*}
x_{n}=\sum\limits_{j\in N_{n}}u_{n}e\left(i_{j}\right),
\end{equation*}
for $n\in\mathbb{N}$, applying \eqref{d2} and \eqref{pow1b}, we get
\begin{eqnarray*}
\rho_{\varphi}^{E}(x_{n})&=&\left\Vert\varphi\left(x_{n}\right)\right\Vert_{E}=\left\Vert\sum\limits_{j\in N_{n}}\varphi\left(u_{n}\right)e\left( i_{j}\right)\right\Vert_{E}\\
&\leq&C_{E}\left\Vert\sum\limits_{j\in N_{n}\backslash\left\{k_{n}\right\} }\varphi\left(u_{n}\right)e\left(i_{j}\right)\right\Vert_{E}+C_{E}\left\Vert\varphi\left(u_{n}\right)e\left(i_{k_{n}}\right)\right\Vert_{E}\\
&\leq&\frac{C_{E}}{C_{E}^{n+2}2^{n+2}}\left\Vert\sum\limits_{j\in N_{n}\backslash\left\{k_{n}\right\}}\varphi\left(\left(1+\frac{1}{n}\right)u_{n}\right)e\left( i_{j}\right)\right\Vert_{E}+\frac{d\cdot C_{E}}{d\cdot C_{E}^{n+2}2^{n+2}}\\
&\leq&\frac{1}{C_{E}^{n+1}2^{n+2}}+\frac{1}{C_{E}^{n+1}2^{n+2}}=\frac{1}{C_{E}^{n+1}2^{n+1}}.
\end{eqnarray*}
Taking $x=\sum\limits_{n=1}^{\infty}x_{n}$, basing on Theorem 1.1 from \cite{ma04}, we get $\varphi\left(x\right)\in E$ and 
\begin{equation*}
\rho_{\varphi}^{E}(x)=\left\Vert\varphi\left(x\right)\right\Vert_{E}=\left\Vert\sum\limits_{n=1}^{\infty}\varphi\left(x_{n}\right)\right\Vert_{E}\leq C_{E}\sum\limits_{n=1}^{\infty }C_{E}^{n}\left\Vert\varphi\left(x_{n}\right)\right\Vert_{E}\leq\sum\limits_{n=1}^{\infty }\frac{C_{E}^{n+1}}{C_{E}^{n+1}2^{n+1}}=\frac{1}{2}.
\end{equation*}
Moreover, for each $\lambda>1$ there is a number $n_{\lambda}\in\mathbb{N}$ such that $\lambda>1+\frac{1}{n_{\lambda}}$, whence by (\ref{pow1b}),
\begin{equation*}
\rho_{\varphi}^{E}(\lambda x)\geq\left\Vert\varphi\left(\left(1+\frac{1}{n_{\lambda}}\right)x\right)\right\Vert_{E}\geq\left\Vert\varphi\left( \left(1+\frac{1}{n_{\lambda}}\right)x_{n_{\lambda }}\right)\right\Vert_{E}>1.
\end{equation*}%
Thus $\left\Vert x\right\Vert_{\varphi}=1$. Defining the sequence $\left(y_{m}\right)_{m=1}^{\infty}$ and the operator $P:l_{\infty}\rightarrow
E_{\varphi}$ like in case $\left(i\right)$, we conclude that $P$ is an order linear isometry.
\end{proof}

\begin{theorem}\label{tko4}
Suppose $E$ is a quasi-Banach ideal space over the counting measure space $\left(\mathbb{N},2^{\mathbb{N}},m\right)$ such that
$E\backslash l_{\infty}\neq\emptyset$.

$\left(i\right)$ If $b_{\varphi}<\infty$, then $E_{\varphi }$ contains an order linearly isometric copy of $l_{\infty}$.

$\left(ii\right)$ Assume that $b_{\varphi}=\infty $ and there exist a number $d\in\left(0,1\right)$ and an increasing sequence of natural numbers $\left(i_{j}\right)_{j=1}^{\infty}$ such that $\lim_{j\rightarrow\infty}\left\Vert e\left(i_{j}\right)\right\Vert_{E}=0$ and $\frac{\left\Vert e\left( i_{j+1}\right)\right\Vert_{E}}{\left\Vert e\left( i_{j}\right) \right\Vert _{E}}\geq d$ for each $j\in\mathbb{N}$, where $e\left( i\right) $ is the $i$-th unit vector. If $\varphi\notin\Delta_{2}\left(\infty\right)$, then $E_{\varphi}$ contains an order linearly isometric copy of $l_{\infty}$.
\end{theorem}

\begin{remark}\label{jjdr}{\rm
$\left(i\right)$ It is easy to see that $E\backslash l_{\infty }\neq\emptyset$ if and only if there exists an increasing sequence of natural
numbers $\left(i_{j}\right)_{j=1}^{\infty}$ such that $\lim_{j\rightarrow\infty}\left\Vert e\left(i_{j}\right)\right\Vert_{E}=0$ (in one direction we should use Theorem 1.1 form \cite{ma04}).

$\left(ii\right)$ The assumption of the statement $(ii)$ of Theorem \ref{tko4} is satisfied by several sequence spaces such as: the Cesàro sequence spaces $ces_{p}$, $1<p<\infty $ (see \cite{kkm} for the respective definition) and the weighted sequence spaces $l_{p}\left(w\right)$, $p>0$, where $w\left(i\right)=\frac{1}{a^{i}}$, $a>1$ or $w\left(i\right)=\frac{1}{i^{q}}$, $q>0.$
}\end{remark}

\begin{proof}
$\left(i\right)$ Suppose $b_{\varphi}<\infty$. Let $u_{n}=\frac{2n+1}{2n+2}b_{\varphi }$. Since $E\backslash l_{\infty }\neq\emptyset$, so we
find an increasing sequence of natural numbers $\left(i_{n}\right)_{n=1}^{\infty}$ such that 
\begin{equation*}
\varphi\left(u_{n}\right)\left\Vert e\left(i_{n}\right)\right\Vert_{E}\leq\frac{1}{C_{E}^{n+1}2^{n+1}}.
\end{equation*}
Define
\begin{equation*}
x_{n}=u_{n}e\left(i_{n}\right)\hspace{3mm}\text{  and  }\hspace{3mm}x=\sum\limits_{n=1}^{\infty}x_{n}.
\end{equation*}
Following analogously as in the proof of Theorem \ref{tko1}, we get $\rho_{\varphi}^{E}(x)\leq\frac{1}{2}$ and $\left\Vert x\right\Vert_{\varphi}=1$. In consequence, taking the sequence $\left(y_{m}\right)_{m=1}^{\infty}$ like in the proof of Theorem \ref{tko1} (see formulas \eqref{yyy} and \eqref{yyyy}) we conclude that the operator $P:l_{\infty}\rightarrow E_{\varphi }$, defined by $P(z)=\sum\limits_{m=1}^{\infty}z_{m}y_{m}$ for $z=\left( z_{m}\right)\in l_{\infty}$, is an order linear isometry.

$\left(ii\right)$ Assume that $b_{\varphi}=\infty$ and $\varphi\notin\Delta_{2}\left(\infty\right)$. Thus we find a number $u_{1}>0$ such that
\begin{equation*}
\varphi\left(u_{1}\right)\left\Vert e\left(i_{1}\right)\right\Vert_{E}\geq1\hspace{3mm}\text{ and }\hspace{3mm}\varphi\left(\left(1+1\right) u_{1}\right)>\frac{2^{2}C_{E}^{2}}{d}\varphi\left(u_{1}\right).
\end{equation*}
Let $j_{1}>1$ be the smallest natural number for which 
\begin{equation*}
\varphi\left(u_{1}\right)\left\Vert e\left( i_{j_{1}}\right)\right\Vert_{E}\leq\frac{1}{2^{2}C_{E}^{2}}.
\end{equation*}
Since $\varphi\left(u_{1}\right)\left\Vert e\left(i_{j_{1}-1}\right)\right\Vert_{E}>\frac{1}{2^{2}C_{E}^{2}}$ and $\left\Vert e\left(i_{j_{1}}\right) \right\Vert_{E}\geq d\left\Vert e\left(i_{j_{1}-1}\right)\right\Vert_{E}$, it follows that
\begin{equation*}
\varphi\left(u_{1}\right)\left\Vert e\left(i_{j_{1}}\right)\right\Vert_{E}>\frac{d}{2^{2}C_{E}^{2}}.
\end{equation*}
Take $u_{2}$ satisfying 
\begin{equation*}
\varphi\left(u_{2}\right)\left\Vert e\left(i_{j_{1}}\right)\right\Vert_{E}\geq1\hspace{3mm}\text{ and }\hspace{3mm}\varphi\left(\left(1+\frac{1}{2}\right)u_{2}\right)>\frac{2^{3}C_{E}^{3}}{d}\varphi\left(u_{2}\right).
\end{equation*}
Take the smallest natural number $j_{2}>j_{1}$ such that 
\begin{equation*}
\varphi\left(u_{2}\right)\left\Vert e\left(i_{j_{2}}\right)\right\Vert_{E}\leq\frac{1}{2^{3}C_{E}^{3}}.
\end{equation*}
Following similarly as for $j_{1}$ we get
\begin{equation*}
\varphi\left(u_{2}\right)\left\Vert e\left(i_{j_{2}}\right)\right\Vert_{E}>\frac{d}{2^{3}C_{E}^{3}}.
\end{equation*}

Proceeding in a such a way by the induction we can find an increasing sequence $\left(u_{n}\right)$ of positive numbers and an increasing
sequence $\left(i_{j_{n}}\right)$ of natural numbers such that
\begin{equation}
\varphi\left(\left(1+\frac{1}{n}\right)u_{n}\right)>\frac{2^{n+1}C_{E}^{n+1}}{d}\varphi\left(u_{n}\right)\label{brak-d2}
\end{equation}
and
\begin{equation*}
\frac{d}{2^{n+1}C_{E}^{n+1}}<\varphi\left( u_{n}\right)\left\Vert e\left(i_{j_{n}}\right)\right\Vert_{E}\leq\frac{1}{2^{n+1}C_{E}^{n+1}}
\end{equation*}
for each $n\in\mathbb{N}$. Defining
\begin{equation*}
x_{n}=u_{n}e\left(i_{j_{n}}\right)\hspace{3mm}\text{ and }\hspace{3mm}x=\sum\limits_{n=1}^{\infty}x_{n},
\end{equation*}
analogously in the proof of Theorem \ref{tko1} we get $\rho_{\varphi}^{E}(x)\leq\frac{1}{2}$. Moreover, for each $\lambda>1$
there exists $n_{\lambda}\in\mathbb{N}$ such that $\lambda>1+\frac{1}{n_{\lambda}},$ whence, by (\ref{brak-d2})
\begin{eqnarray*}
\rho_{\varphi}^{E}(\lambda x)&\geq&\left\Vert\varphi\left(\left(1+\frac{1}{n_{\lambda }}\right)x\right)\right\Vert_{E}\geq\varphi\left(\left(1+\frac{1}{n_{\lambda}}\right)u_{n_{\lambda}}\right)\left\Vert e\left(i_{j_{n_{\lambda}}}\right)\right\Vert_{E}\\
&>&\frac{2^{n+1}C_{E}^{n+1}}{d}\cdot\varphi\left(u_{n_{\lambda }}\right)\left\Vert e\left(i_{j_{n_{\lambda}}}\right)\right\Vert_{E}>\frac{
2^{n+1}C_{E}^{n+1}}{d}\cdot\frac{d}{2^{n+1}C_{E}^{n+1}}=1.
\end{eqnarray*}
Thus $\left\Vert x\right\Vert_{\varphi }=1$. Defining the sequence $\left(y_{m}\right)_{m=1}^{\infty}$ and the operator $P:l_{\infty }\rightarrow E_{\varphi }$ as in case $\left( i\right) $, we conclude that $P$ is an order linear isometry.
\end{proof}

\begin{theorem}\label{tocc}
Let $E$ be a quasi-Banach ideal space over the counting measure space $\left(\mathbb{N},2^{\mathbb{N}},m\right)$ and $\varphi\in\Delta _{2}^{E}$. Then $E_{\varphi}$ is order continuous if and only if $E$ is order continuous.
\end{theorem}

\begin{proof}
In the case when $l_{\infty }\subset E$ or (neither $E\subset l_{\infty}$ nor $l_{\infty }\subset E$) the proof goes the same way as the proof of
Theorem \ref{toc}. Now we will consider the case $E\subset l_{\infty}$.

Suppose $E$ is order continuous. Thus $E\subset c_{0}$. Take any $x\in E_{\varphi}$ and a sequence $\left(x_{n}\right)$ in $\left(E_{\varphi
}\right)_{+}$ such that $0\leftarrow x_{n}\leq\left\vert x\right\vert$. By the definition of $E_{\varphi },$ there exists a number $\lambda_{x}>0$ such that $\varphi\left(\lambda_{x}\left\vert x\right\vert\right)\in E$, whence $\varphi\left(\lambda_{x}\left\vert x\right\vert\right)\in c_{0}$. Take $\lambda>0$. There exists a number $m=m\left(\lambda\right)\in\mathbb{N}$ such that $\lambda\leq 2^{m}\lambda_{x}$. Since $\varphi\in\Delta_{2}^{E}=\Delta _{2}\left(0\right)$, so $a_{\varphi }=0$ and $\left\vert x\right\vert\in c_{0}$. Denote by $u_{0}$, $K$ the constants from the condition $\Delta_{2}\left(0\right)$, note that without loss of generality we can assume that $2u_{0}<b_{\varphi}$. In consequence, we can find a number $i_{\lambda}\in\mathbb{N}$ such that $2^{m}\lambda_{x}\left\vert x\left( i\right)\right\vert\leq2u_{0}<b_{\varphi }$ for each $i>i_{\lambda}$. Set $N_{\lambda}=\mathbb{N}\backslash\left\{1,\ldots,i_{\lambda }\right\}$. Then
\begin{equation*}
\varphi\left(\lambda\left\vert x\right\vert\chi_{N_{\lambda}}\right)\leq\varphi\left( 2^{m}\lambda _{x}\left\vert x\right\vert\chi_{N_{\lambda}}\right)\leq K^{m}\varphi\left(\lambda _{x}\left\vert x\right\vert\chi_{N_{\lambda }}\right)\in E.
\end{equation*}
Since $0\leftarrow\varphi\left(\lambda\left\vert x_{n}\right\vert\chi_{N_{\lambda}}\right)\leq\varphi\left(\lambda\left\vert x\right\vert
\chi_{N_{\lambda }}\right)\in E$ and $E$ is order continuous, it follows that
\begin{equation*}
\lim_{n\rightarrow\infty}\rho_{\varphi}^{E}(\lambda x_{n}\chi_{N_{\lambda}})=\lim_{n\rightarrow\infty}\left\Vert\varphi\left(\lambda\left\vert x_{n}\right\vert\chi_{N_{\lambda }}\right)\right\Vert_{E}=0.
\end{equation*}
Moreover, $\lambda\left\vert x_{n}\left(i\right)\right\vert\rightarrow0$ for every $i=1,2,\ldots,i_{\lambda }$, whence $\varphi\left(\lambda\left\vert x_{n}\left(i\right)\right\vert\right)\rightarrow0$ and $\varphi\left(\lambda\left\vert x_{n}\left( i\right)\right\vert\right)\linebreak\|e(i)\|_{E}\rightarrow0$ for each $i=1,2,\ldots,i_{\lambda}$. Consequently, $\lim_{n\rightarrow \infty }\rho _{\varphi }^{E}(\lambda
x_{n})=0.$ Since $\lambda >0$ was arbitrary, applying Lemma \ref{le33}, we conclude that $\left\Vert x_{n}\right\Vert_{\varphi}\rightarrow0$.

Assume now that $E$ is not order continuous. Then there exists $x\in E$ and a sequence $\left(x_{n}\right) $ in $E_{+}$ such that $0\leftarrow
x_{n}\leq\left\vert x\right\vert$ for each $n\in \mathbb{N}$ and $\left\Vert x_{n}\right\Vert_{E}\not\rightarrow0$. Without loss of generality we can assume that $\left\Vert x_{n}\right\Vert_{E}>1$ for each $n\in\mathbb{N}$. Define $y=\varphi^{-1}\left(x\right) $ and $y_{n}=\varphi ^{-1}\left(x_{n}\right)$ for any $n\in\mathbb N$. We have $0\leftarrow y_{n}\leq y$.

If $b_{\varphi }=\infty $ or $b_{\varphi }<\infty $ and $\varphi\left(b_{\varphi}\right)=\infty$, then 
\begin{equation}
\rho_{\varphi}^{E}(y_{n})=\left\Vert\varphi\left(y_{n}\right)\right\Vert_{E}=\left\Vert x_{n}\right\Vert_{E}>1  \label{jjn}
\end{equation}
for each $n\in\mathbb{N}$ and consequently $\left\Vert y_{n}\right\Vert_{\varphi}>1$ for each $n\in\mathbb{N}$. Thus $E_{\varphi }$ is not order continuous.

Assume now that $b_{\varphi}<\infty$ and $\varphi\left(b_{\varphi }\right)<\infty$. Take $n\in\mathbb{N}$. If $x_{n}\left(i\right)\leq\varphi\left(b_{\varphi}\right)$ for any $i\in\mathbb{N}$, then analogously as in \eqref{jjn} we get $\rho_{\varphi}^{E}(y_{n})>1$ and $\left\Vert y_{n}\right\Vert_{\varphi}>1$. If $x_{n}\left(i_{n}\right) >\varphi\left(b_{\varphi}\right)$ for some $i_{n},$ then $\left\Vert y_{n}\right\Vert_{\varphi}\geq1$. Then again $E_{\varphi}$ is not order continuous.
\end{proof}

Applying Theorems \ref{tko3}, \ref{tko4} and \ref{tocc} we get the following

\begin{corollary}
$\left(i\right)$ Let $E$ be a quasi-Banach symmetric sequence space with $E\neq l_{\infty}$. Then $E_{\varphi}$ is order continuous if and only if $E$ is order continuous and  $\varphi\in\Delta_{2}\left( 0\right) $.

$\left(ii\right) $ Let $E$ be the Ces\`{a}ro sequence space $ces_{p}$ with $1<p<\infty$ or the weighted sequence spaces $l_{p}\left(w\right)$ for $p\in(0,\infty)$, where $w\left(i\right)=\frac{1}{i^{q}}$ with $p\cdot q\in (0,1]$. Then $E_{\varphi}$ is order continuous if and only if  $\varphi\in\Delta_{2}\left(\mathbb{R}\right)$.

$\left(iii\right)$ Let $E$ be the  weighted sequence space $l_{p}\left(w\right)$ for $p\in(0,\infty)$, where $w\left(i\right)=\frac{1}{a^{i}}$ with $a>1$, or $w\left(i\right)=\frac{1}{i^{q}}$ with $p\cdot q>1$. Then $E_{\varphi}$ is order continuous if and only if $\varphi\in\Delta_{2}\left({\infty}\right)$.
\end{corollary}

Denote by $P$ the property of having the linear order isometric copy of $l_{\infty}$. Until now we have shown that if $\varphi\notin\Delta_{2}^{E}$, then $E_{\varphi}\in\left(P\right)$ (under some natural assumptions on $E$), see Corollary \ref{cko1} and Theorems \ref{tko3} and \ref{tko4}. Now we are going to consider if the following implications hold:

$\left(1\right)$ if $E\in\left(P\right)$, then $E_{\varphi }\in \left(P\right)$?

$\left(2\right)$ if $E_{\varphi}\in\left(P\right)$, then $E\in\left(P\right)$?\\
We will find some natural conditions which imposed on the function $\varphi$ guarantee that the above implications are true. Note that in such a way we can see a characteristic dichotomy: to prove that $E_{\varphi}\in\left(P\right)$ we are able to find either the suitable conditions imposed on $\varphi $ or the suitable conditions imposed on $E$ which ensure that $E_{\varphi}\in\left(P\right)$.
\vspace{3mm}

The proof of the next theorem proceed the same as proof Theorem 1 in \cite{hud}.

\begin{theorem}\label{tk1} A quasi-Banach ideal space $E$ contains an order linearly isometric copy of $l_{\infty}$ if and only if there exists in $E$ a sequence $(x_{n})_{n=1}^{\infty}$ of positive elements which are pairwise orthogonal, that is, $\supp x_{n}\cap\supp x_{m}=\emptyset$ if $n\neq m$, such that $\|x_{n}\|_{E}=1$ for any $n\in\mathbb N$ and $\|x\|_{E}=1$, where $x:=\sum_{n=1}^{\infty}x_{n}\in E$.\end{theorem}

Now we recall definition of the condition $\Delta_{\varepsilon }$ which has been introduced in \cite{chkk2}*{Definition 2.6}.

We say that an Orlicz function $\varphi$ satisfies the condition $\Delta_{\varepsilon}$ for all $u\in\mathbb R_{+}$ ($\varphi\in\Delta_{\varepsilon}(\mathbb{R}_{+})$ for short) if for any $\varepsilon\in (0,1)$ there exists $\delta=\delta\left(\varepsilon\right)\in (0,1)$ such that the inequality
\begin{equation}\label{rrds}
\varphi(\varepsilon u)\leq\delta\varphi(u)
\end{equation}
holds for any $u\geq0$. We say that $\varphi$ satisfies the condition $\Delta_{\varepsilon}$ at infinity [at zero] ($\varphi\in\Delta_{\varepsilon}(\infty)$ [$\varphi\in\Delta_{\varepsilon}(0)$] for short) if for any $\varepsilon\in(0,1)$ there exist $\delta=\delta\left(\varepsilon\right)\in(0,1)$ and $u_{0}=u_{0}(\varepsilon)>0$ such that inequality \eqref{rrds} holds for any $u\geq u_{0}$ [$0\leq u\leq u_{0}$], respectively.
\vspace{2mm}

For any quasi-Banach ideal space $E$ and any Orlicz function $\varphi $ we say that $\varphi $ satisfies condition $\Delta_{\varepsilon }^{E}$ ($\varphi \in \Delta _{\varepsilon }^{E}$ for short) if:\\
$(1)$ $\varphi\in\Delta_{\varepsilon}(\mathbb{R}_{+})$ whenever neither $L_{\infty }\subset E$ nor $E\subset L_{\infty }$,\\
$(2)$ $\varphi\in\Delta_{\varepsilon}(\infty )$ whenever $L_{\infty}\subset E$,\\
$(3)$ $\varphi\in\Delta_{\varepsilon}(0)$ whenever $E\subset L_{\infty }$.

\begin{remark}\label{rwde}{\rm
$(i)$ Clearly, if $\varphi$ is convex (or s-convex for some $s\in(0,1]$), then $\varphi\in\Delta_{\varepsilon}(\mathbb{R}_{+})$ automatically. Moreover, if $\varphi\in\Delta_{\varepsilon}(\mathbb{R}_{+})$, then $\varphi$ must be strictly increasing on the interval $(a_{\varphi },b_{\varphi})$. 

$(ii)$ If $\varphi\in\Delta_{\varepsilon}(0)$, then $\varphi$ satisfies this condition for any $u_{1}>a_{\varphi}$ so long as $\varphi$ is strictly increasing on the interval $(a_\varphi,u_{1})$. Let us take any fixed $\varepsilon\in(0,1)$. 

$(ii-a)$ First assume that $0=a_{\varphi}<u_{0}<u_{1}$, where $u_{0}=u_{0}(\varepsilon)$ is the constant from the condition $\Delta_{\varepsilon}(0)$, and $\varphi$ is strictly increasing on the interval $(0,u_{1})$.

If $\varphi(u_{1})<\infty$, then the function
\begin{equation}\label{wde}
f_{\varepsilon}(u):=\frac{\varphi(\varepsilon u)}{\varphi(u)}
\end{equation}
is continuous and its values are smaller than 1 on the interval $[u_{0},u_{1}]$. Hence $\delta_{f}=\delta_{f}(\varepsilon):=\sup_{u\in[u_{0},u_{1}]}f_{\varepsilon}(u)<1$. Denoting $\delta_{1}=\delta_{1}(\varepsilon):=\max(\delta,\delta_{f})$, where $\delta=\delta(\varepsilon)$ is the constant from the condition $\Delta_{\varepsilon}(0)$, we get
\begin{equation}\label{wde1}
\varphi(\varepsilon u)\leq\delta_{1}\varphi(u)
\end{equation}
for any $u\in[0,u_{1}]$.

Let now $u_{1}=b_{\varphi}<\infty$, $\varphi(u_{1})=\varphi(b_{\varphi})=\infty$ and $u_{2}\in(\max(u_{0},\varepsilon u_{1}),u_{1})$. Then for any $u\in(u_{2},u_{1}]$ we have
\begin{equation*}
\varphi(\varepsilon u)\leq\varphi(\varepsilon u_{1})=\frac{\varphi(\varepsilon u_{1})}{\varphi(u_{2})}\varphi(u_{2})\leq\frac{\varphi(\varepsilon u_{1})}{\varphi(u_{2})}\varphi(u).
\end{equation*}
Simultaneously, $\delta_{f}=\delta_{f}(\varepsilon):=\sup_{u\in[u_{0},u_{2}]}f_{\varepsilon}(u)<1$. So, defining $\delta_{1}=\delta_{1}(\varepsilon)=:\max(\delta,\delta_{f},$ $\frac{\varphi(\varepsilon u_{1})}{\varphi(u_{2})})$, where, as above, $\delta=\delta(\varepsilon)$ is the constant from the condition $\Delta_{\varepsilon}(0)$, we obtain inequality \eqref{wde1} for any $u\in[0,u_{1}]$.

$(ii-b)$ Now suppose that $0<a_{\varphi}<u_{1}$ and $\varphi$ is strictly increasing on the interval $(a_\varphi,u_{1})$. Define $u_{3}=u_{3}(\varepsilon):=a_{\varphi}/\varepsilon$. 

In the case when $u_{1}\leq u_{3}$, we obtain immediately inequality \eqref{wde1} for any $\delta_{1}\in(0,1)$ and any $u\in[0,u_{1}]$. 

If $u_{3}<u_{1}$ and $\varphi(u_{1})<\infty$, then defining function $f_{\varepsilon}$ by formula \eqref{wde} on the interval $[u_{3},u_{1}]$, we get inequality \eqref{wde1} for any $u\in[0,u_{1}]$ with the constant $\delta_{1}=\delta_{1}(\varepsilon):=\sup_{u\in[u_{3},u_{1}]}f_{\varepsilon}(u)$.

Finally assume that $u_{3}<u_{1}=b_{\varphi}<\infty$, $\varphi(u_{1})=\varphi(b_{\varphi})=\infty$ and $u_{2}\in(\max(u_{3},\varepsilon u_{1}),u_{1})$. Proceeding analogously as above, we obtain inequality \eqref{wde1} for any $u\in[0,u_{1}]$ with the constant $\delta_{1}=\delta_{1}(\varepsilon):=\max\left(\sup_{u\in[u_{3},u_{2}]}f_{\varepsilon}(u),\frac{\varphi(\varepsilon u_{1})}{\varphi(u_{2})}\right)$.

$(iii)$ Analogously, if $\varphi\in\Delta_{\varepsilon}(\infty)$, then $\varphi$ satisfies this condition for any $0<u_{1}<b_{\varphi}$ so long as $\varphi$ is strictly increasing on the interval $(u_{1},b_{\varphi})$. Let $\varepsilon\in(0,1)$ and $u_{3}=u_{3}(\varepsilon):=a_{\varphi}/\varepsilon$.

$(iii-a)$ Suppose that $u_{1}<u_{0}<b_{\varphi}=\infty$, where $u_{0}=u_{0}(\varepsilon)$ is the constant from the condition $\Delta_{\varepsilon}(\infty)$, and $\varphi$ is strictly increasing on the interval $(u_{1},\infty)$.

If $\varphi(u_{1})>0$, then $\delta_{f}=\delta_{f}(\varepsilon):=\sup_{u\in[u_{1},u_{0}]}f_{\varepsilon}(u)<1$, where the function $f_{\varepsilon}$ is defined by \eqref{wde}. Denoting $\delta_{1}=\delta_{1}(\varepsilon)=:\max(\delta,\delta_{f})$, where $\delta=\delta(\varepsilon)$ is the constant from the condition $\Delta_{\varepsilon}(\infty)$, we get 
\begin{equation}\label{wde3}
\varphi(\varepsilon u)\leq\delta_{1}\varphi(u)
\end{equation}
for any $u\in[u_{1},\infty)$.

Let now $\varphi(u_{1})=0$, that is, $u_{1}=a_{\varphi}>0$. Then we get inequality \eqref{wde3} for any $u\in[u_{1},\infty)$ with the constant $\delta_{1}=\delta_{1}(\varepsilon)=\delta$ whenever $u_{0}\leq u_{3}$ and $\delta_{1}=\delta_{1}(\varepsilon):=\max(\delta,\sup_{u\in[u_{3},u_{0}]}f(u))$ in the opposite case.

$(iii-b)$ Finally assume that $u_{1}<b_{\varphi}<\infty$ and $\varphi$ is strictly increasing on the interval $(u_{1},b_{\varphi})$.

If $\varphi(b_{\varphi})<\infty$, then we obtain inequality \eqref{wde3} for any $u\in[u_{1},\infty)$ with the constant $\delta_{1}=\delta_{1}(\varepsilon):=\sup_{u\in[u_{1},b_{\varphi}]}f_{\varepsilon}(u)$ whenever $\varphi(u_{1})>0$, $\delta_{1}=\delta_{1}(\varepsilon):=\sup_{u\in[u_{3},b_{\varphi}]}f_{\varepsilon}(u)$ in the case when $\varphi(u_{1})=0$ and $u_{3}<b_{\varphi}$, and for any $\delta_{1}\in(0,1)$ whenever $\varphi(u_{1})=0$ and $b_{\varphi}\leq u_{3}$.

Let now $\varphi(b_{\varphi})=\infty$. If $\varphi(u_{1})>0$, then we get
\begin{equation*}
\varphi(\varepsilon u)\leq\varphi(\varepsilon b_{\varphi})=\frac{\varphi(\varepsilon b_{\varphi})}{\varphi(u_{2})}\varphi(u_{2})\leq\frac{\varphi(\varepsilon b_{\varphi})}{\varphi(u_{2})}\varphi(u)
\end{equation*}
for any $u\in(u_{2},b_{\varphi}]$, where $u_{2}\in(\max(u_{1},\varepsilon b_{\varphi}),b_{\varphi})$. In consequence, we get inequality \eqref{wde3} for any $u\in[u_{1},\infty)$ with the constant $\delta_{1}=\delta_{1}(\varepsilon):=\max\left(\sup_{u\in[u_{1},u_{2}]}f_{\varepsilon}(u),\frac{\varphi(\varepsilon b_{\varphi})}{\varphi(u_{2})}\right)$.

While, if $\varphi(u_{1})=0$, that is, $0<a_{\varphi}=u_{1}$, we obtain inequality \eqref{wde3} for any $u\in[u_{1},\infty)$ with the constant $\delta_{1}=\delta_{1}(\varepsilon):=\max\left(\sup_{u\in[u_{3},u_{2}]}f_{\varepsilon}(u),\frac{\varphi(\varepsilon b_{\varphi})}{\varphi(u_{2})}\right)$ ($u_{2}\in(\max(u_{3},\varepsilon b_{\varphi}),b_{\varphi})$)  whenever $u_{3}<b_{\varphi}$, and for any $\delta_{1}\in(0,1)$ otherwise.

$(iv) $ From the above consideration we conclude immediately that if $\varphi\in\Delta_{\varepsilon}(0)$, $\varphi\in\Delta_{\varepsilon}(\infty)$ and $\varphi$ is strictly increasing on the interval $(a_{\varphi},b_{\varphi})$, then $\varphi\in\Delta_{\varepsilon}(\mathbb{R}_{+})$. In particular, if $a_{\varphi}>0$, $b_{\varphi}<\infty$ and $\varphi$ is strictly increasing on the interval $(a_{\varphi},b_{\varphi})$, then $\varphi\in\Delta_{\varepsilon}(\mathbb{R}_{+})$.

$(v)$ Finally, notice that if $\alpha_{\varphi}^{E}>0$ with the constant $K=1$, then $\varphi\in\Delta_{\varepsilon}^{E}$ (it is enough to take $\delta \left( \varepsilon\right)=\varepsilon^{p}$).
}\end{remark}

\begin{example}\label{de}{\rm
$(i)$ First, let us notice that the functions $\varphi_{1}$ and $\varphi_{2}$ defined in Example \ref{ne} do not satisfy conditions $\varphi\in\Delta_{\varepsilon}(\infty)$ and $\varphi\in\Delta_{\varepsilon}(0)$, respectively. Obviously, the both function are strictly increasing on $\mathbb{R}_{+}$.

$(ii)$ Take $u_{0}=0$, $u_{1}=1$ and $u_{n}=2u_{n-1}=2^{n-1}$ for each $n>1$. Moreover, let $\varphi_{3}\left(0\right)=0$, $\varphi_{3}\left(u_{n}\right)=n$ for each $n\geq1$ and define the function $\varphi_{3}$ to be continuous in $\mathbb{R}_{+}$ and linear in each interval $\left[u_{n-1},u_{n}\right]$ for $n\in\mathbb N$. Clearly, $\varphi_{3}$ is strictly increasing on the interval $(0,\infty)$. Furthermore, for $a=1/2^{m}$ where $m\in\mathbb N$, we obtain that
\begin{equation*}
\lim_{n\rightarrow\infty}\frac{\varphi_{3}\left(a u_{n}\right)}{\varphi_{3}\left(u_{n}\right)}=\lim_{n>m;n\rightarrow\infty}\frac{\varphi_{3}\left(u_{n-m}\right)}{\varphi_{3}\left(u_{n}\right)}=\lim_{n>m;n\rightarrow\infty}\frac{n-m}{n}=1.
\end{equation*}
Consequently, $\varphi_{3}\notin\Delta_{\varepsilon}(\infty)$ and $\alpha_{\varphi_{3}}^{\infty}=0$.

$(iii)$ Let $u_{1}=1$ and $u_{n}=\frac{u_{n-1}}{2}=\frac{1}{2^{n-1}}$ for any $n>1$, $\varphi_{4}(0)=0$, $\varphi_{4}(u)=u$ for $u\geq1/2$, $\varphi_{4}(u_{n})=1/n$ for $n\geq3$ and $\varphi_{4}$ is continuous in $\mathbb{R}_{+}$ and linear in each interval $\left[u_{n},u_{n-1}\right]$ for $n\geq3$. Obviously, $\varphi_{4}$ is strictly increasing on the interval $(0,\infty)$. Moreover, for $a=1/2^{m}$ where $m\in\mathbb N$, we conclude that
\begin{equation*}
\lim_{n\rightarrow\infty}\frac{\varphi_{4}\left(a u_{n}\right)}{\varphi_{4}\left(u_{n}\right)}=\lim_{n\rightarrow\infty}\frac{\varphi_{4}\left(u_{n+m}\right)}{\varphi_{4}\left(u_{n}\right)}=\lim_{n\rightarrow\infty}\frac{n}{n+m}=1.
\end{equation*}
Hence $\varphi_{4}\notin\Delta_{\varepsilon}(0)$ and $\alpha_{\varphi_{4}}^{0}=0$.
}\end{example}

\begin{lemma}\label{mod-nor=1}
Let $x\in E_{\varphi}$ be such that $\rho_{\varphi}^{E}(x)=1$. Moreover, suppose one of the following three conditions holds:
\begin{itemize}
\item[$(i)$] $\varphi\in\Delta_{\varepsilon}^{E}$, whenever neither $L_{\infty}\subset E$ nor $E\subset L_{\infty}$,
\item[$(ii)$] $\varphi\in\Delta_{\varepsilon}(\mathbb R_{+})$ or $(\varphi\in\Delta_{\varepsilon}^{E}$, $\varphi$ is strictly increasing on the interval $(a_{\varphi},b_{\varphi})$ and there exists a constant $B>0$ such that $\left\vert x\left(t\right)\right\vert\geq B$ for $\mu $-a.e.\ $t\in T)$, whenever $L_{\infty}\subset E$,
\item[$(iii)$] $\varphi\in\Delta_{\varepsilon}^{E}$ and $\varphi$ is strictly increasing on the interval $(a_{\varphi},\min(\varphi^{-1}(1/a_{E}),b_{\varphi}))$, where $a_{E}$ is defined by formula \eqref{stala}, whenever $E\subset L_{\infty}$.
\end{itemize}
Then $\Vert x\Vert_{\varphi}=1$.
\end{lemma}

\begin{proof} Let $x\in E_{\varphi}$ and $\rho_{\varphi}^{E}(x)=1$. First suppose that $\varphi\in\Delta_{\varepsilon}(\mathbb R_{+})$. Let $\lambda\in(0,1)$ be fixed. If $\varphi(|x|/\lambda)\notin E$, then $\rho_{\varphi}^{E}(x/\lambda)=\infty$. In the opposite case there exists $\delta=\delta\left(\lambda\right)\in(0,1)$ such that
\begin{equation}
1=\left\Vert\varphi\left(\frac{\lambda|x|}{\lambda}\right)\right\Vert_{E}\leq\left\Vert\delta\varphi\left(\frac{|x|}{\lambda}\right)\right\Vert _{E}=\delta\left\Vert\varphi\left(\frac{\left\vert x\right\vert}{\lambda }\right)\right\Vert_{E}. \label{mod1}
\end{equation}
In consequence, $\rho_{\varphi}^{E}(x/\lambda)>1$ for any $\lambda \in (0,1)$, whence $\Vert x\Vert_{\varphi }=1$.

Now assume  that $L_{\infty}\subset E$, $\varphi\in\Delta_{\varepsilon}^{E}$ ($\varphi\in\Delta_{\varepsilon}(\infty)$), $\varphi$ is strictly increasing on the interval $(a_{\varphi},b_{\varphi})$ and there exists a constant $B>0$ such that $\left\vert x\left(t\right)\right\vert\geq B$ for $\mu $-a.e.\ $t\in T$. By Remark \ref{rwde}$(iii)$, $\varphi\in\Delta_{\varepsilon}(\infty)$ with the constant $u_{1}(\varepsilon)=B$ for any $\varepsilon\in(0,1)$. Thus, for any $\lambda\in (0,1)$ there exists $\delta=\delta\left(\lambda\right)\in (0,1)$ such that inequality (\ref{mod1}) is true again, so $\Vert x\Vert_{\varphi}=1$.

Finally suppose $E\subset L_{\infty}$, $\varphi\in\Delta_{\varepsilon}^{E}$ and $\varphi$ is strictly increasing on the interval $(a_{\varphi},\min\linebreak(\varphi^{-1}(1/a_{E}),b_{\varphi}))$. Without loss of generality, we can assume that $\varphi\in\Delta_{\varepsilon}(\mathbb R_{+})$. 
Indeed, if $b_{\varphi}<\infty$, $\varphi(b_{\varphi})<\infty$ and $\varphi^{-1}(1/a_{E})=b_{\varphi}$, then $\varphi\in\Delta_{\varepsilon}(\mathbb R_{+})$ (see Remark \ref{rwde}). In the opposite case, that is if $\varphi^{-1}(1/a_{E})<b_{\varphi}$, defining new Orlicz function $\psi$, by $\psi(u)=\varphi(u)$ for $u\in[0,\varphi^{-1}(1/a_{E})]$ and $\psi(u)=u-(\varphi^{-1}(1/a_{E})-1/a_{E})$ for $u>\varphi^{-1}(1/a_{E})$, we get $\psi\in\Delta_{\varepsilon}(\mathbb R_{+})$ and $(E_{\varphi},\|\cdot\|_{\varphi})\equiv(E_{\psi},\|\cdot\|_{\psi})$ (see Remarks \ref{r42-5} and \ref{rwde}). Proceeding analogously as in (\ref{mod1}), we obtain again $\Vert x\Vert_{\varphi}=1$.\end{proof}

\begin{theorem}\label{ptt}
Suppose one of the following three conditions holds:
\begin{itemize}
\item[$(i)$] $\varphi\in\Delta_{\varepsilon}^{E}$, whenever neither $L_{\infty}\subset E$ nor $E\subset L_{\infty}$,
\item[$(ii)$] $\varphi\in\Delta_{\varepsilon}(\mathbb R_{+})$ or $(\varphi\in\Delta_{\varepsilon}^{E}$, $\varphi$ is strictly increasing on the interval $(a_{\varphi},b_{\varphi})$ and $E$ is rearrangement invariant Banach space $($see \cites{bs,ltf}$)$ over $T=\left( 0,1\right)$ or $T=\left( 0,\infty \right)$ with $\mu$ being the Lebesgue measure such that $\supp E_{a}=\supp E)$,  whenever $L_{\infty}\subset E$,
\item[$(iii)$] $\varphi\in\Delta_{\varepsilon}^{E}$ and $\varphi$ is strictly increasing on the interval $(a_{\varphi},\min(\varphi^{-1}(1/a_{E}),b_{\varphi}))$, where $a_{E}$ is defined by formula \eqref{stala}, whenever $E\subset L_{\infty}$.
\end{itemize}
If $E$ contains an order linearly isometric copy of $l_{\infty }$, then $E_{\varphi }$ contains also such a copy.
\end{theorem}

\begin{proof}
By Theorem \ref{tk1}, there exists a sequence $(x_{n})_{n=1}^{\infty }$ in $S(E_{+})=S(E)\cap E_{+}$ with $\supp x_{n}\cap \supp x_{m}=\emptyset $ for $n\neq m$ such that $\lVert \sum_{n=1}^{\infty }x_{n}\rVert _{E}=1$.

First we claim that, in the case $L_{\infty }\subset E$ with $E$ being rearrangement invariant Banach space (see condition $(ii)$)  we can assume without loss of generality that there is a constant $B>0$ such that $\left\vert x_{n}\left( t\right)\right\vert\geq B$ for all $n\in\mathbb{N}$ and for $\mu $-a.e.\ $t\in T$. Set $x=\sum_{n=1}^{\infty }x_{n}.$ Denote also 
\begin{equation*}
d_{E}\left( x,E_{a}\right)=\inf\left\{\left\Vert x-y\right\Vert_{E}:y\in E_{a}\right\}.
\end{equation*}
Then $d_{E}\left( x,E_{a}\right)=1$ (see the proof of Theorem 2.1 in \cite{kkt}). Moreover, by Theorem 3.5 and Corollary 3.7 in \cite{kkt}, 
\begin{equation*}
1=d_{E}\left(x,E_{a}\right)=d_{E}\left(x^{\ast},E_{a}\right)=\lim\limits_{n\rightarrow\infty}\lVert x^{\ast}\chi _{T_{n}}\rVert_{E},
\end{equation*}
where $T_{n}=T\cap\left((0,\frac{1}{n})\cup(n,\infty)\right) $ and $x^{\ast }$ is the nonincreasing rearrangement of $x$ (see \cites{bs,ltf}). If $x^{\ast}\left(\infty\right)>0$, then applying Theorem 2 from \cite{hud} (and its proof) we can build the respective copy of $l_{\infty}$ consisting of pairwise orthogonal elements satisfying the claim with $B=x^{\ast}\left(\infty\right)$. Suppose now $x^{\ast}\left(\infty\right)=0$. Then 
\begin{equation*}
\lVert x^{\ast}\chi_{(n,\infty)}\rVert_{E}\leq A\lVert x^{\ast}\chi_{(n,\infty)}\rVert_{L_{\infty}}\rightarrow0
\end{equation*}
as $n\rightarrow\infty$. Thus $\lim\limits_{n\rightarrow \infty }\lVert x^{\ast}\chi_{(0,\frac{1}{n})}\rVert_{E}=1$, whence $d_{E}\left( x^{\ast
}\chi_{\left(0,b\right)},E_{a}\right)=1$ for some $b>0$ such that $x^{\ast }\left(b\right)>0$. Applying again Theorem 2 from \cite{hud} (and
its proof) we can build the respective copy of $l_{\infty}$ consisting of pairwise orthogonal elements satisfying the claim with $B=x^{\ast}\left(b\right)$.

Take $f_{n}=\varphi ^{-1}\left( x_{n}\right) $. Assume that $\varphi \left(b_{\varphi}\right)=\infty$. Then, by Lemma \ref{skladanie},  $\rho _{\varphi}^{E}\left(f_{n}\right)=1$ for any natural $n$, whence, by Lemma \ref{mod-nor=1}, $\Vert f_{n}\Vert_{\varphi}=1$ for the same $n$. Moreover, $\supp(f_{n})\cap\supp(f_{m})=\emptyset$ for $n\neq m$ and $\lVert\sum_{n=1}^{\infty}f_{n}\rVert_{\varphi}=\rho_{\varphi }^{E}\left(\sum_{n=1}^{\infty}f_{n}\right)=1$. Applying again Theorem \ref{tk1} for the space $E_{\varphi }$ we finish the proof.

Now let $\varphi \left( b_{\varphi }\right) <\infty .$ Denote 
\[A=\left\{ t\in T:\left\vert x_{n}\left(t\right)\right\vert>\varphi \left(b_{\varphi}\right)\right\}.\]
If $\mu\left(A\right)=0$, then we follow as above. Suppose $\mu\left(A\right)>0$. Then 
\[\varphi\left(\varphi^{-1}\left(\left\vert x_{n}\left(t\right)\right\vert\right)\right)=\varphi\left(b_{\varphi}\right)<\left\vert x_{n}\left(t\right) \right\vert\text{ for }t\in A.\]
In consequence, $\varphi\left(f_{n}\right)\leq x_{n}$, whence $\rho_{\varphi}^{E}\left(f_{n}\right)\leq1$. Moreover, $\rho_{\varphi}^{E}\left( f_{n}/\lambda\right)=\infty$ for each $0<\lambda<1$. Thus $\Vert f_{n}\Vert_{\varphi}=1$ and we can finish as above again.
\end{proof}
\vspace{3mm}

We say that an Orlicz function $\varphi$ satisfies the condition $\Delta_{2-str}$ for all $u\in\mathbb R_{+}$  ($\varphi\in\Delta_{2-str}(\mathbb{R}_{+})$ for short) if for any $\varepsilon>0$ there exists $\delta=\delta\left(\varepsilon\right)>0$ such that the inequality
\begin{equation}\label{rds}
\varphi((1+\delta)u)\leq(1+\varepsilon)\varphi(u)
\end{equation}
holds for any $u\in\mathbb R_{+}$. We say that $\varphi$ satisfies the condition $\Delta_{2-str}$ at infinity [at zero] ($\varphi\in\Delta_{2-str}(\infty)$ [$\varphi\in\Delta_{2-str}(0)$] for short) if for any $\varepsilon >0$ there exist $\delta=\delta\left(\varepsilon\right)>0$
and $u_{0}=u_{0}(\varepsilon)>0$ such that $\varphi(u_{0})<\infty$ [$\varphi(u_{0})>0$] and inequality \eqref{rds} holds for any $u\geq u_{0}$ [$0\leq u\leq u_{0}$], respectively (see \cite{chkk}*{Definition 2.15}).

For any quasi-Banach ideal space $E$ and any Orlicz function $\varphi$ we say that $\varphi$ satisfies condition $\Delta^{E}_{2-str}$ ($\varphi\in\Delta^{E}_{2-str}$ for short) if:\\
$(1)$ $\varphi\in\Delta_{2-str}(\mathbb R_{+})$ whenever neither $L_{\infty}\subset E$ nor $E\subset L_{\infty}$,\\
$(2)$ $\varphi\in\Delta_{2-str}(\infty)$ whenever $L_{\infty}\subset E$,\\
$(3)$ $\varphi\in\Delta_{2-str}(0)$ whenever $E\subset L_{\infty}$.

\begin{remark}{\rm
Clearly, if $\varphi \in \Delta _{2-str}^{E}$, then $\varphi \in \Delta_{2}^{E}$. Moreover, for a convex Orlicz function conditions $\Delta _{2-str}
$ and $\Delta _{2}$ are equivalent (see \cite{c}*{Theorem 1.13, p. 9}). However, for non-convex Orlicz functions they need not be equivalent (see Example below Definition 2.15 in \cite{chkk}).
}\end{remark}

\begin{lemma}\label{nor-mod=1}
Suppose one of the following three conditions holds:
\begin{itemize}
\item[$(i)$] $\varphi\in\Delta_{2-str}^{E}$, whenever neither $L_{\infty}\subset E$ nor $E\subset L_{\infty}$,
\item[$(ii)$] $\varphi\in\Delta_{2-str}(\mathbb R_{+})$ or $(E$ is a Banach ideal space, $\varphi\in\Delta_{2-str}^{E}$ and $\varphi$ is strictly increasing on the interval $(a_{\varphi},\infty))$, whenever $L_{\infty}\subset E$,
\item[$(iii)$] $\varphi\in\Delta_{2-str}^{E}$,  $1/a_{E}\leq\varphi(b_{\varphi})$ and $\varphi$ is strictly increasing on the interval $(0,\varphi^{-1}(1/a_{E}))$, where $a_{E}$ is defined by formula \eqref{stala},  whenever $E\subset L_{\infty}$.
\end{itemize}
Then for any $x\in E_{\varphi}$ such that $\Vert x\Vert_{\varphi}=1$, we have $\rho_{\varphi }^{E}(x)=1$.
\end{lemma}

\begin{proof} We will show that if $\rho_{\varphi}^{E}(x)<1$, then $\|x\|_{\varphi}<1$. First assume $\varphi\in\Delta_{2-str}(\mathbb R_{+})$.  Then for $\varepsilon>0$ satisfying the condition $\left(1+\varepsilon\right)\rho_{\varphi }^{E}(x)=1$ there exists $\delta>0$ such that
\begin{equation}\label{rd}
\rho_{\varphi}^{E}(\left(1+\delta\right)x)=\left\Vert\varphi\left(\left(1+\delta\right)\left\vert x\right\vert\right)\right\Vert_{E}\leq\left\Vert(1+\varepsilon)\varphi \left(\left\vert x\right\vert\right)\right\Vert_{E}=(1+\varepsilon)\rho_{\varphi}^{E}(x)=1.
\end{equation}
Thus $\left\Vert x\right\Vert_{\varphi}\leq\frac{1}{1+\delta }<1$.

Let now $L_{\infty}\subset E$, $E$ be a Banach ideal space, $\varphi\in\Delta_{2-str}^{E}$ and  $\varphi$ is strictly increasing on the interval $(a_{\varphi},\infty)$. Take $\varepsilon>0$ and $B>a_{\varphi}$ such that
\begin{equation*}
\varphi\left(2B\right)\left\Vert\chi_{T}\right\Vert_{E}+\left(1+\varepsilon\right)\rho_{\varphi}^{E}(x)\leq1.
\end{equation*}
Applying Lemma 4.14 from \cite{chkk} (note that this lemma is true, if we assume that $\varphi$ is strictly increasing on the interval $(a_{\varphi},\infty)$ in place of $\mathbb{R}_{+}$ and we replace the phrase {\em "for any $u_{1}>0$"} with {\em "for any $u_{1}>a_{\varphi}$"}) we find a number $\delta=\delta\left(\varepsilon,B\right)$ such that $\varphi ((1+\delta )u)\leq(1+\varepsilon )\varphi (u)$ for all $u\geq B$. We may assume without loss of generality that $\delta\in\left(0,1\right)$. Denote
\begin{equation*}
A_{1}=\left\{t\in T:\left\vert x\left(t\right)\right\vert\geq B\right\}\text{ and }A_{2}=\left\{t\in T:\left\vert x\left(t\right)\right\vert<B
\right\} .
\end{equation*}
Then
\begin{eqnarray*}
\rho_{\varphi }^{E}(\left(1+\delta\right)x)&=&\left\Vert\varphi\left(\left(1+\delta\right)\left\vert x\right\vert\right)\right\Vert_{E}
\leq\left\Vert\varphi\left(\left(1+\delta\right)\left\vert x\right\vert\right)\chi _{A_{1}}\right\Vert_{E}
+\left\Vert\varphi\left(\left(1+\delta\right)\left\vert x\right\vert\right)\chi_{A_{2}}\right\Vert _{E}\\
&\leq&\left(1+\varepsilon\right)\left\Vert\varphi\left( \left\vert x\right\vert\right)\right\Vert_{E}
+\varphi\left(2B\right)\left\Vert\chi_{T}\right\Vert_{E}\leq1.
\end{eqnarray*}
So, we obtain again $\left\Vert x\right\Vert_{\varphi}\leq\frac{1}{1+\delta}<1$.

Finally assume $E\subset L_{\infty}$. By Remark \ref{r42-5}, we have $|x(t)|\leq\varphi^{-1}(1/a_{E})$ for $\mu$-a.e.\ $t\in T$. Let $\varepsilon>0$ be such that $\left(1+\varepsilon \right)\rho_{\varphi }^{E}(x)=1$. By Lemma 4.15 from \cite{chkk}  we find a number $\delta=\delta\left(\varepsilon,\varphi^{-1}(1/a_{E})\right)$ such that $\varphi ((1+\delta )u)\leq(1+\varepsilon )\varphi (u)$ for all $u\leq\varphi^{-1}(1/a_{E})$ (if necessary, we can define new Orlicz function $\psi$, by $\psi(u)=\varphi(u)$ for $u\in[0,\varphi^{-1}(1/a_{E})]$ and $\psi(u)=u-(\varphi^{-1}(1/a_{E})-1/a_{E})$ for $u>\varphi^{-1}(1/a_{E})$, then we have $\psi\in\Delta_{2-str}(\mathbb R_{+})$ and $(E_{\varphi},\|\cdot\|_{\varphi})\linebreak\equiv(E_{\psi},\|\cdot\|_{\psi})$). Proceeding analogously as in \eqref{rd} we get $\left\Vert x\right\Vert_{\varphi}<1$.
\end{proof}

\begin{theorem}\label{tt}
Suppose one of the following three conditions holds:
\begin{itemize}
\item[$(i)$] $\varphi\in\Delta_{2-str}^{E}$, whenever neither $L_{\infty}\subset E$ nor $E\subset L_{\infty}$,
\item[$(ii)$] $\varphi\in\Delta_{2-str}(\mathbb R_{+})$ or $(E$ is a Banach ideal space, $\varphi\in\Delta_{2-str}^{E}$ and $\varphi$ is strictly increasing on the interval $(a_{\varphi},\infty))$, whenever $L_{\infty}\subset E$,
\item[$(iii)$] $\varphi\in\Delta_{2-str}^{E}$,  $1/a_{E}\leq\varphi(b_{\varphi})$ and $\varphi$ is strictly increasing on the interval $(0,\varphi^{-1}(1/a_{E}))$, where $a_{E}$ is defined by formula \eqref{stala},  whenever $E\subset L_{\infty}$.
\end{itemize}
If $E_{\varphi }$ contains an order linearly isometric copy of $l_{\infty }$, then $E$ contains also such a copy.
\end{theorem}

\begin{proof}
By Theorem \ref{tk1}, there exists a sequence $(x_{n})_{n=1}^{\infty}$ in $S(\left(E_{\varphi }\right)_{+})=S(E_{\varphi})\cap\left(E_{\varphi}\right)_{+}$ with $\supp x_{n}\cap\supp x_{m}=\emptyset$ for $n\neq m$ such that $\lVert\sum_{n=1}^{\infty}x_{n}\rVert_{\varphi}=1$. Take $f_{n}=\varphi\left(x_{n}\right)$ for $n\in\mathbb N$. By Lemma \ref{nor-mod=1}, $f_{n}\in S(E_{+})=S(E)\cap E_{+}$ for the same $n$ and $\lVert\sum_{n=1}^{\infty}f_{n}\rVert_{E}=\rho_{\varphi}^{E}\left(\sum_{n=1}^{\infty}x_{n}\right)=1$. Moreover, $\supp(f_{n})\cap\supp(f_{m})=\emptyset$ for $n\neq m$. Applying again Theorem \ref{tk1} for the space $E$ we finish the proof.
\end{proof}

\begin{example}{\rm
Let $\varphi\left(u\right)=u^{p}$, $p>0.$ Clearly, $\varphi\in\Delta_{\varepsilon }\left(\mathbb{R}_{+}\right)$ and $\varphi\in\Delta_{2-str}\left(\mathbb{R}_{+}\right)$. Let $\left(E_{1},\left\Vert\cdot\right\Vert_{E_{1}}\right)=\left(L_{\psi},\left\Vert\cdot\right\Vert_{\psi}\right)$ (see Example \ref{olo}). By Theorem \ref{tko1} $\left(E_{1},\left\Vert\cdot\right\Vert_{E_{1}}\right)$ contains an order linearly isometric copy of $l_{\infty}$, whence applying Theorem \ref{ptt}, $\left(E_{1}\right)_{\varphi}$ contains also such a copy. On the other hand, the space $\left(E_{2}\right)_{\varphi }$, where $\left(E_{2},\left\Vert\cdot\right\Vert_{E_{2}}\right)=\left(L_{\psi},\left\Vert\cdot\right\Vert_{\psi }^{O}\right)$ - see Example \ref{olo}, does not contain such a copy. Indeed, otherwise, in view of Theorem \ref{tt}, the space $\left(E_{2},\left\Vert\cdot\right\Vert_{E_{2}}\right)$ would contain such a copy, which is a contradiction with Theorem 3 in \cite{cch}.}
\end{example}

\section{Acknowledgements}

Scientific research of Pawe\l\ Kolwicz was carried out at the Poznan University of Technology as part of the Rector's grant 2021 (research
project No. 0213/SIGR/2154).

\section{References}
\label{ll1}

\end{document}